\documentclass[11pt]{article}
\usepackage{verbatim} 
\usepackage{graphicx}
\usepackage[utf8]{inputenc}
\usepackage{amsmath,amsfonts,amssymb,amscd,amsthm,txfonts,hyperref}
\usepackage{xcolor}
\newtheorem{theorem}{Theorem}[section]
\newtheorem{proposition}[theorem]{Proposition}
\newtheorem{lemma}[theorem]{Lemma}
\newtheorem{corollary}[theorem]{Corollary}
\newtheorem{remark}{Remark}[section]
\newtheorem{theo}{Theorem}

\theoremstyle{definition}
\newtheorem{definition}[theorem]{Definition}

\newtheorem{assum}{Assumption}

\newcommand{\R}{\mathbb{R}}
\newcommand{\C}{\mathbb{C}}
\newcommand{\N}{\mathbb{N}}
\newcommand{\Z}{\mathbb{Z}}
\newcommand{\T}{\mathbb{T}}

\newcommand{\lap}{\bigtriangleup}

\newcommand{\E}{\mathbb E}
\renewcommand{\Re}{\mbox{Re}}
\newcommand{\du}{``du"}

\newcommand{\ent}[1]{\lfloor #1\rfloor}

\newcommand{\an}[1]{\langle #1 \rangle}

\newcommand{\grad}{\bigtriangledown}

\begin{document}

\title{Invariance of Gibbs measures under the flows of Hamiltonian equations on the real line}
\author{Federico Cacciafesta\footnote{Dipartimento di Matematica, Universit$\grave{\text{a}}$ degli studi di Padova,  Via Trieste, 63, 35131 Padova PD - Italy. -  
email: cacciafe@math.unipd.it}, Anne-Sophie de Suzzoni\footnote{Universit\'e Paris 13, Sorbonne Paris Cit\'e, LAGA, CNRS ( UMR 7539), 99, avenue Jean-Baptiste Cl\'ement, F-93430 Villetaneuse, France - email: adesuzzo@math.univ-paris13.fr}}

\maketitle

\begin{abstract} 
We prove that the Gibbs measures $\rho$ for a class of Hamiltonian equations written as \begin{equation}\label{theeq}\partial_t u = J (-\lap u + V'(|u|^2)u)\end{equation} on the real line are invariant under the flow of \eqref{theeq} in the sense that there exist random variables $X(t)$ whose laws are $\rho$ (thus independent from $t$) and such that $t\mapsto X(t)$ is a solution to \eqref{theeq}. Besides, for all $t$, $X(t)$ is almost surely not in $L^2$ which provides as a direct consequence the existence of global weak solutions for initial data not in $L^2$. The proof uses Prokhorov's theorem, Skorohod's theorem, as in the strategy in \cite{burqtzv} and Feynman-Kac's integrals.
\end{abstract}

\tableofcontents

\section{Introduction and setting}

\subsection{Introduction}

We prove the invariance of Gibbs measures on $\R$ under the flow of Hamiltonian equations using Feynman-Kac's theorem.

The problem is the following. We have a Hamiltonian equation that writes 
\begin{equation}\label{hamileq2}
 \partial_t u = J \grad_{\overline u }H(u)
\end{equation}
where $J$ is a skew symmetric (or anti Hermitian) operator and 
$$
H(u) = -\frac12 \int \overline u \lap u + \frac12\int V(|u|^2)
$$
is the Hamiltonian of the equation and displays a purely kinetic part $-\frac12 \int \overline u \lap u$ and a potential one $\frac12\int V(|u|^2)$. The equation \eqref{hamileq} can be written as
\begin{equation}\label{hamileq}
\partial_t u = J(-\lap u + V'(|u|^2)u).
\end{equation}
Under these assumptions, the mass $M(u) = \frac12 \int |u|^2$ is conserved under the flow of \eqref{hamileq}. We assume that the equation is defocusing in the sense that $V$ is non negative.

The type of equation that we have is mind is the non linear Schr\"odinger equation on $\R$ in the case when $u$ is complex valued and the modified Korteweg de Vries equation when $u$ is real valued.

We prove that the Gibbs measure $e^{-H(u) - M(u)} \du$ is invariant under the flow of \eqref{hamileq}.

The study of Gibbs measures and their invariance under the flows of Hamiltonian equations with related consequences on deterministic well-posedness of nonlinear models represents a very active research field with very large literature, and an attempt to give a complete picture of it is out of the scope of this manuscript. The interest began with the seminal paper by Lebowitz, Rose and Speer \cite{LebRosSpe}, that first constructed and studied the equilibrium of Gibbs measures associated to some nonlinear Schr\"odinger equation with periodic boundary conditions, and was subsequently enhanced by the many contributions by Bourgain, that first suggested, among the other things, the connection between existence of Gibbs measures and the related deterministic dynamics. In his pioneering paper \cite{bou1} he proved in particular a "strong invariance" of the Gibbs measure $\rho$ associated to the  nonlinear Schr\"odinger equation on the 1-dimensional torus, meaning that the flow $\psi(t)$ of the  equation is $\rho$-
almost surely globally well-posed and for all $\rho$ measurable set $A$ and for all times $t \in \R$,
$$
\rho ( \psi(t)^{-1} (A)) = \rho(A).
$$
See also \cite{BouBul,nahstaff,ORT,ohtzv,staffilani,Zhidkov} for related results. 
Among the possible applications of the existence of a strongly invariant measures, we mention the fact (remarked by Bourgain) that they can be used as "conserved quantities" to globalize solutions to the corresponding PDE defined only for local times. The strategy for building invariant measures, at its core, consists in looking at a PDE as an infinite dimensional Hamiltonian system, approaching it with a sequence of finite dimensional ones, use Liouville's Theorem to get finite dimensional invariance and then pass to the limit. A key role in this argument is played by the Fourier transform, that being discrete in the compact setting allows, after truncating, to define the "natural" sequence of finite dimension approaching equations. Therefore, the natural problem of extending this kind of results to the non compact setting, that has subsequently been addressed by many authors (see \cite{bouinf,benoh,randoh1,luhmen,randoh2} and references therein) is, in general, remarkably more difficult. In \cite{
burthotzve} the authors proved invariance under the flow of the Schr\"odinger equation with a quadratic potential on $\R$, using the fact that $-\lap + |x|^2$ has a discrete spectrum and so, in a way, bypassing the problem of infinite volume. We also mention  \cite{mckeanvan,xu} in which the finite speed of propagation is exploited to deal with the wave equation, and \cite{CdS1,ansokg} in which the non linearity is localised. We also stress the fact that in dimension $2$ or higher these problems present more difficulties (see for instance \cite{bourgaindim2,BouBul,deng,OhTho}), as the invariant measure is supported on spaces for which no good control on the flows is available, and therefore the connection with the deterministic dynamics seems to be weaker. We should also mention the papers \cite{burq1,burq2}, in which the authors suggest a further application of invariant measures, using them to prove supercritical well-posedness for some nonlinear wave equation.


%
%
%
%

We do not hope to achieve a general result such as the "strong invariance" discussed above for our generic equation \eqref{hamileq} on $\R$. What we prove is the following theorem.

\begin{theo}\label{teo1}
 Under Assumptions \ref{assum-V}, \ref{assum-J} on $J$ and $V$, there exist a non-trivial measure $\rho$ (independent from $t$), a probability space $(\Omega,\mathcal A, P)$ and a random variable $X_\infty$ with values in $\mathcal C( \R, \mathcal D')$ such that \begin{itemize}
\item for all $t\in \R$, the law of $X_\infty(t)$ is $\rho$,
\item $X_\infty$ is a weak solution (in the sense of the distributions) of \eqref{hamileq}.
\end{itemize}
What is more, $X_\infty(t)$ is almost surely an s-H\"older continuous map, for $s<\frac12$, and the law of $X_\infty(t,x)$ is absolutely continuous with respect to the Lebesgue measure and independent from $x$ and $t$.
\end{theo}

\begin{remark} The properties of $X_\infty$ are consequences of properties of $\rho$ and ensure that, for any fixed $t$, $X_\infty$ is almost surely not in $L^2$. Indeed, as $X_\infty$ is H\"older continuous, if it is in $L^2$, then $X_\infty(t,x)$ converges towards $0$ when $x$ goes to $\infty$. And since the law of $X_\infty(t,x)$ does not depend on $x$, the probability that it converges towards $0$ at $\infty$ is less than the probability for $X_\infty(t,0)$ to be $0$ which is null since the law of $X_\infty(t,0)$ is absolutely continuous with respect to the Lebesgue measure.
 Notice that of course the solution defined by $X_\infty$ is (possibly) not unique, as for two configurations of the probability space $\omega_1\neq \omega_2 \in \Omega$ such that $X_\infty(t=0,\omega_1) = X_\infty(t=0,\omega_2)$ the solutions $X_\infty(t,\omega_1)$ and $X_\infty(t,\omega_2)$ might not be equal.
 \end{remark}

\begin{remark}
This result can be deduced for the Schr\"odinger equation from the paper by Bourgain, \cite{bouinf} who proves a stronger theorem in the case of a cubic non linearity, since he proves not only the existence of a weak flow but also its uniqueness. A strong invariance result can be deduced from it. The idea is to take the invariant measure on a box of size $L$ with periodic boundary condition and pass to the limit. On the other hand, our strategy allows us to obtain results for a wider class of equations that in fact include, after some additional manipulations, also some variable coefficients Schr\"odinger equations  (see Appendix). Note that the theorem also provides global existence of weak solution for a large class of semi-linear equations in dimension 1 for non localised data, including a large choice of dispersion relation (thanks to the freedom on $J$) and nonconvex nonlinearities.
\end{remark}

\begin{remark}\label{varcoeffrk}
Theorem \ref{teo1} and its proof can be in fact modified to deal with some variable coefficients equations as well. We will explain the necessary modifications to the proof in Appendix \ref{appvar}.
\end{remark}

The strategy of our proof is inspired by \cite{burqtzv}, in which the authors adapt to the context of dispersive PDEs a technology already developed in fluid mechanics that essentially relies on the application of Prokhorov's and Skorohod's Theorems. The idea is to construct a sequence of random variables which solve some approximating equations for which the existence of an invariant measure is standard to prove and then pass to the limit. This will produce the existence of a measure and a random variable as in Theorem \ref{teo1}. The main difficulty in the present context is due to the infinite volume setting, which makes the approximating procedure significantly less intuitive, together with the infinite speed of propagation. Nevertheless, we show that the only invariance we need is the one of a finite dimensional problem on a dilated torus, and is obtained just by the application of Liouville's Theorem for finite dimensional Hamiltonian flows. The rest is reduced to proving that the measure $\rho$ is the 
limit of the invariant measures for the finite dimensional problems along with some probabilistic estimates. The idea of the proof is the following.

We take $L >0$ and build the Gibbs measure for the ODE
\begin{equation}\label{fineq}
\partial_t u = \Pi_{N(L)} J_L \Pi_{N(L)} \grad_{\overline u} H_L(u)
\end{equation}
where $\Pi_{N(L)}$ projects onto the Fourier modes in $[-N(L),N(L)]$, and $J_L$ is a smooth periodisation of $J$. We set
$$
H_L(u) = -\frac12 \int_{2\pi L\T} \overline u \lap u + \frac12 \int_{2\pi L\T} \chi_LV(|u|^2)
$$
where $\chi_L$ is a smooth compactly supported function. we call $\psi_L$ the flow of \eqref{fineq}.

The Gibbs measure is given by
$$
d\rho_L(u) = Z^{-1} e^{-2H_L(u)+M(u)} d\mathcal{L}(u)
$$
where $\mathcal{L}$ is the Lebesgue measure and $Z$ is a normalization factor ($\rho$ is a probability measure). We point out that this expression is not formal, as indeed the measure is defined on a finite-dimensional space. This measure can also be written
$$
d\rho_L(u) = Z_L^{-1} e^{-\int \chi_L V(|u|^2)} d\mu_L(u)
$$
where $Z_L$ is a normalization factor and $\mu_L$ is the measure induced by the random variable
$$
\xi_L^f(x) = \sum_{k\in \Z\cap [-N(L)L,N(L)L]} \frac{e^{i kx/L}}{\sqrt{1+ \frac{k^2}{L^2}}} \Big( W\Big( \frac{k+1}{L}\Big) - W \Big( \frac{k}{L}\Big)\Big)
$$
and $W$ is a complex valued Brownian motion. It is a finite-dimensional Gaussian measure as the sum is finite. Letting $N$ go to $\infty$ independently from $L$, we get that this random variable converges in some sense to 
$$
\xi_L(x) = \sum_{k\in \Z} \frac{e^{i kx/L}}{\sqrt{1+ \frac{k^2}{L^2}}} \Big( W\Big( \frac{k+1}{L}\Big) - W \Big( \frac{k}{L}\Big)\Big),
$$
and if we let $L$ go to $\infty $ in $\xi_L$ we get that it converges towards
$$
\xi(x) =  \int \frac{e^{i kx}}{\sqrt{1+ k^2}} dW(k)
$$
which is a known object called the oscillatory or Ornstein-Uhlenbeck process, we refer to \cite{SIMfunint} or \cite{GlimJaf}. It induces a measure $\mu$ on functions.

Hence, if we take the limit only in the kinetic part of the Hamiltonian we get the measure
$$
d\rho_{L,2} (u) = Z_{L,2}^{-1} e^{- \int \chi_L V(|u|^2)} d\mu(u)
$$
where $Z_{L,2}$ is a normalization factor. If we let $\chi_L$ go to the function constant to $1$, we get thanks to Feynman-Kac's theory a non trivial measure $\rho$, which is described precisely in the book by Simon, \cite{SIMfunint} pages 58 and onward.

The idea is that by choosing $N(L)$ and $\chi_L$ appropriately then the sequence $\rho_L$ converges weakly towards $\rho$. This, together with a moment bound uniform in $L$, which is needed to guarantee that the limiting random variable is a weak solution of the equation, is heuristically sufficient to get the result.

Indeed, we then build $\nu_L$ which is the image measure $\rho_L$ under the flow $\psi_L(t)$ of \eqref{fineq}. That means that $\nu_L$ is the law of a random variable $X_L$ such that $X_L(t) = \psi_L(t) X_L(0)$ and such that the law of $X_L(0)$ is $\rho_L$. Thanks to the Prokhorov-Skorohod method, we can reduce the problem to proving that the family $(\nu_L)_L$ is tight in $\mathcal C (\R,H_\varphi)$, for some Banach space $H_\varphi$. The choice of $H_\varphi$ is not so important, it just has to be separable in order to apply the Prokhorov-Skorohod method.  The topology in time, though, has to be such that taking $X_\infty(t)=\displaystyle \lim_{L\rightarrow+\infty} X_L(t)$ makes sense, that is why we choose $\mathcal C(\R)$. This method has been used on dispersive equation in \cite{burqtzv,OhTho}, and comes from the fluid mechanics literature, see for example \cite{machin1,machin2}.

Using the invariance of $\rho_L$ under $\psi_L$, we then reduce the problem to proving estimates on $\rho_L$ and to proving that $\rho_L$ goes to $\rho$ (and not to something trivial). These results are consequences of Feynman-Kac's theory. 

The paper is organized as follows : in the next subsection, we give or recall definitions and notations, together with some preliminary probabilistic properties. We give the assumptions on $J,V, \chi_L,N(L)$, and others.

In Section 2, we explain the Prokhorov-Skorohod method and reduce our problem to proving estimates on $\rho_L$ and its convergence towards $\rho$.

In Section 3, we prove the estimates and the convergence relying on our choices for $\chi_L$ and $N(L)$.

\subsection{Assumptions and notations}

We write $\an x = \sqrt{1 + x^2}$ and $D = \sqrt{1 - \partial_x^2}$.

\paragraph{Assumptions on the equation}

\begin{assum}\label{assum-V} One chooses $V$ in $\mathcal C^2$ such that there exist $C$, $r_V$, such that for all $u\in \C$,
\begin{eqnarray}\label{assuV}
0\leq V(|u|^2) &\leq & C\an{u}^{r_V}, \\
|V'(|u|^2)| &\leq & C\an{u}^{r_V} ,\\
|V''(|u|^2) |& \leq & C\an{u}^{r_V}.
\end{eqnarray}

One also requires that the operator $-\lap + |x|^2 + V(|x|^2)$ has a non-degenerate first eigenvalue, which should often be the case, see \cite{ReeSim4}.
\end{assum}

One may choose $r_V > 1$.

\begin{assum}\label{assum-J} One chooses $J$ skew-symmetric such that there exist $\kappa \in \R^+$, $C\geq 0$, such that for all $s\in ]0,\frac12[$, $u \in L^2(\R)$,
$$
\|D^{s-\kappa+1/4} J (1-\lap) u\|_{L^2} \leq C \|u\|_{L^2}.
$$
and such that for all $\sigma \geq 0$, all $u \in H^{\sigma + \kappa}$
$$
\|D^\sigma J u \|_{L^2} \leq C \|u\|_{H^{\sigma + \kappa}}.
$$
We also assume that if $u$ is $\mathcal C^\infty$ with compact support, then $Ju$ also is and $Ju$ has the same support. We choose $\kappa$ big enough such that $s-\kappa \in 2\Z$.

We set for some test function $u$, $J_L u(x) = \eta_L(x) J(\eta_L u)(x) $ if $x\in [-L,L]$ and $J_Lu(x) = J_L u(x-\ent{x/L}L)$ otherwise where $\eta_L$ is a $\mathcal C^\infty$ function equal to $1$ on $[-L+1,L-1]$ and to $0$ outside $[L,L]$. This defines $J_L$ which inherits the properties on $J$, except the last one.

\end{assum}

We have in mind $J= \partial_x$, which corresponds to take mKdV, or $J = i$ which corresponds to NLS, but one may choose $J = \sum_{k\leq \kappa} a_k(x) \partial_x^k$ with $a_k$ $\mathcal C^\infty$ bounded functions whose derivatives are also bounded as long as $J$ remains skew-symmetric, and this corresponds to some variable coefficients PDEs (see the Appendix). Notice that our assumptions \ref{assum-V} allow to include any polynomial-type nonlinearity.

\paragraph{Notations on measures}

Let $W(k)$ be a centered complex Gaussian process defined on $\R$ with covariance
$$
\E(\overline{W(k)} W(l) ) = \delta_{kl \geq 0} \min(|k|,|l|)
$$
where $\delta_{kl \geq 0} = 1$ if $k$ and $l$ have the same sign and $\delta_{kl \geq 0} = 1$ otherwise. This yields that
$$
\E(\overline{dW(k)}dW(l)) = \delta(k-l), \textrm{ and } \E(|W(t) - W(s)|^2) = |t-s|.
$$
For further properties of Gaussian processes, we refer to \cite{simonPphi2}.

For all $L>0$, we write
\begin{equation}\label{defxiL}
\xi_L(x) = \sum_{k\in \Z} \frac{e^{i kx/L}}{\sqrt{1+ \frac{k^2}{L^2}}} \Big( W\Big( \frac{k+1}{L}\Big) - W \Big( \frac{k}{L}\Big)\Big)
\end{equation}
if the solution of the equation has values in $\C$ and
$$
\xi_L(x) = \Re \Big( \sum_{k\in \Z} \frac{e^{i kx/L}}{\sqrt{1+ \frac{k^2}{L^2}}} \Big( W\Big( \frac{k+1}{L}\Big) - W \Big( \frac{k}{L}\Big)\Big)\Big)
$$
if the solution of the equation has values in $\R$.

We write $\xi$ the limit when $L$ goes to $\infty$ of this random variable, that is
$$
\xi (x) = \int \frac{e^{i kx}}{\sqrt{1+ k^2}} dW(k)
$$
in the complex case and 
$$
\xi(x) = \Re \Big( \int \frac{e^{i kx}}{\sqrt{1+ k^2}} dW(k)\Big)
$$
in the real case.

We write $\xi_L^f$ the restriction to low frequencies of $\xi_L$, that is with $N(L)$ a function that goes to $\infty$ when $L$ goes to $\infty$,
$$
\xi_L^f(x) = \Pi_{N(L)} \xi_L(x) = \sum_{k\in \Z\cap [-N(L)L,N(L)L]} \frac{e^{i kx/L}}{\sqrt{1+ \frac{k^2}{L^2}}} \Big( W\Big( \frac{k+1}{L}\Big) - W \Big( \frac{k}{L}\Big)\Big)
$$
in the complex case and we take its real part in the real case.

We write $\mu_L$ the measure induced by $\xi_L^f$, $\mu$ the measure induced by $\xi$, and $\mu_{L,1}$ the one induced by $\xi_L$.

With $R(L)$ a function that goes to $\infty$ when $L$ goes to $\infty$, we write 
\begin{equation}\label{ZL3}
Z_{L,3} = \int e^{-\int_{-R(L)}^R(L) V(|u(x)|^2)dx} d\mu(u)
\end{equation}
and 
$$
d\rho_{L,3} (u) = \frac{ e^{-\int_{-R(L)}^R(L) V(|u(x)|^2)dx}}{Z_{L,3}} d\mu(u).
$$

We also write
\begin{equation}\label{ZL}
Z_L = \int\Big( e^{-\int \chi_L(x) V(|u(x)|^2)dx} \Big)d\mu_L(u),
\end{equation}
and 
$$
d\rho_L (u) =  \frac{e^{-\int \chi_L(x) V(|u(x)|^2)dx}}{Z_L} d\mu_L.
$$

Moreover, we write
\begin{equation}\label{ZL1}
Z_{L,1} = \int\Big( e^{-\int \chi_L(x) V(|u(x)|^2)dx} \Big)d\mu_{L,1}(u),
\end{equation}
$$
d\rho_{L,1}(u) = \frac{e^{-\int \chi_L(x) V(|u(x)|^2)dx}}{Z_{L,1}} d\mu_{L,1}(u).
$$
and
\begin{equation}\label{ZL2}
Z_{L,2} = \int e^{-\int \chi_L(x) V(|u(x)|^2)dx} d\mu(u),
\end{equation}
$$
d\rho_{L,2}(u) = \frac{e^{-\int \chi_L(x) V(|u(x)|^2)dx}}{Z_{L,2}} d\mu(u).
$$

We recall that $\rho$ is the limit when $R$ goes to $\infty$ of 
$$
\frac{ e^{-\int_{-R}^{R} V(|u(x)|^2)dx}}{Z'_{R}} d\mu(u)
$$
where $Z'_R$ is a normalization factor. It exists, is non-trivial, is carried by $s$-H\"older continuous maps for $s<\frac12$ and the law of $u(x)$ induced by $\rho$ is independent from $x$ and absolutely continuous with regard to the Lebesgue measure, see \cite{SIMfunint} pp 58 and onward.

Hence $\rho$ is also the limit of $\rho_{L,3}$ when $L$ goes to $\infty$.

We sum up the notations on measures in the following table

\vspace{0.5cm}
\begin{tabular}{|c|c|c|c|c|}\hline
 & random variable & linear measure & Z & final measure \\\hline
Finite dimension & $\xi_L^f$ & $\mu_L$ & $Z_L$ & $\rho_L$ \\\hline
Including high frequencies & $\xi_L$ & $\mu_{L,1}$ & $Z_{L,1}$ & $\rho_{L,1}$ \\\hline
L goes to $\infty$ & $\xi$ & $\mu $ & $Z_{L,2}$ & $\rho_{L,2}$ \\
in the kinetic energy& & & & \\\hline
$\chi_L \leftarrow 1_{[-R(L),R(L)]}$ & $\xi$ & $\mu$ & $Z_{L,3}$ & $\rho_{L,3}$ \\\hline
$R(L)\rightarrow \infty$ & $\xi$ & $\mu$ &  & $\rho$ \\\hline
\end{tabular}
\vspace{0.5cm}

Finally, we write $\nu_L$ the image measure of $\rho_L$ under the flow $\psi_L$, that is
$$
\nu_L(A) = \rho_L \{ u | \, t\mapsto \psi_L(t) u \in A \}.
$$

\paragraph{Norms.}

We introduce the functional setting that we will be using in the sequel. Let $S_s$ be the space induced by the norm

\begin{equation}\label{defSs}
\|f\|_{S_s}^2 = \|\an{t}^{-1} \varphi D^{s-\kappa} f\|_{L^2(\R^2)}^2 + \|\an{t}^{-1}\varphi D^{s-\kappa} \partial_t f\|_{L^2(\R^2)}^2,
\end{equation}
let $H_\varphi$ be the space induced by the norm
\begin{equation}\label{defHphi}
\|f\|_\varphi = \|\an{x}^{-1} \varphi D^{-\kappa} f\|_{L^2(\R)}
\end{equation}
and $S$ be the space $S = \mathcal C(\R,H_\varphi)$ normed by
\begin{equation}\label{defS}
\|f\|_S^2 = \sup_{t\in \R}\an{t}^{-3} \|f\|^2_\varphi.
\end{equation}

The map $\varphi(|x|)$ is a smooth decreasing on $\R^+$ positive map that we specify later.

\paragraph{Assumptions on $\chi_L$, $N(L)$, $s$.}

\begin{assum}\label{assum-R}Let $R(L)$ be such that for $L\geq 1$,
$$
Z_{L,3} \geq L^{-1/6}.
$$
\end{assum}

This is possible because 
$$
\int e^{-\int_{-R}^{R} V(|u(x)|^2)dx} d\mu(u)
$$
is positive for all $R\geq 0$ and is equal to $1$ if $R=0$ so that the case $L=1$ is included.

\begin{assum}\label{assum-R'} Let $R'(L) = R(L) + \frac1{C\sqrt L}$ where $C$ is a (big) positive constant.\end{assum}

\begin{assum}\label{assum-chi} We assume that $\chi_L$ is a $\mathcal C^\infty$ function such that $\chi_L(x) = 1$ on $[-R(L),R(L)]$, and $\chi_L(x) = 0$ outside $[-R'(L),R'(L)]$ and $\chi_L(x) \in [0,1]$.\end{assum}

Under these assumptions, $\chi_L$ converges to $1$ in $\an x L^\infty$.
\begin{assum}\label{assum-N} Let $N(L) \geq L^4$ and assume $N(L) \geq  L^{1/(3 - 6s)}$ where $s$ is taken according to Assumption \ref{assum-sob}. Assume also that
$$
N(L)^{-1/4} L \max_{R(y) \leq L} y^{1/6}
$$
is uniformly bounded in $L$.\end{assum}

\begin{assum} \label{assum-sob} We take $s< \frac12$ such that the embedding $H^s \hookrightarrow L^p$ holds for $p = 2r_V+2, 2r_V$ and $2r_V+4$.\end{assum}

\paragraph{Invariance.}

\begin{proposition} We have that $\rho_L$ is strongly invariant under the flow $\psi_L(t)$ of 
$$
\partial_t u = - \Pi_{N(L)}J_L \Pi_{N(L)}\lap u + \Pi_{N(L)}J_L \Pi_{N(L)}\chi_L V'(|u|^2) u 
$$
in $H^s(\T_L)$, for all $s<\frac12$. The map $\Pi_{N(L)}$ is the projection onto the Fourier modes in $[-N(L),N(L)]$. In other words, the equation is globally well-posed on a set of full $\rho_L$ measure and for all measurable sets $A$ of $H^s(\T_L)$ and all times $t$ we have 
$$
\rho_L(\psi_L(t)^{-1}(A)) = \rho_L(A).
$$
\end{proposition}

This is due to the fact that we are in finite dimension, thus Liouville's Theorem applies, and $H_L(u)$ is invariant under $\psi_L(t)$. 

\begin{remark}
The equation above admits local solutions because under our assumptions on $V$ and especially that it is $\mathcal C^2$ it is an ODE with Lipschitz right hand side, hence Cauchy-Lipschitz Theorem applies. It is globally well-posed because $M(u)$ is conserved and thus we have a control on any norm of the solution.
\end{remark}

\paragraph{Oscillatory processes and Feynman-Kac.}
A key part of our argument exploits the property that $\mu$ is a complex or real valued oscillatory process process, also known as Ornstein-Uhlenbeck process. Therefore, for the sake of completeness, we devote this paragraph to recall some classical facts about them that will be fundamental in the sequel. First of all, we recall the following definition.
\begin{definition}
A family $\{q(x)\}_{x\in\R}$ of Gaussian random variables is called {\em oscillatory process} if 
$$
\mathbb{E}(q(x)q(y))=\frac12 e^{-|x-y|}.
$$
We will denote with $dq$ the measure on paths $\omega(x)$ associated to the oscillatory process.
\end{definition}
Its law is invariant under translations in $x$. We recall the following important fact: in analogy to what happens with Brownian motions, it is natural (and useful) to link oscillatory processes with suitable semi-groups.
In the following proposition we collect some basic facts about oscillatory processes that will be needed in the sequel.
\begin{proposition}\label{use?}
Let $q(x)$ be an oscillatory process, $L_0=-\frac12\frac{d^2}{dx^2}+\frac12 x^2-\frac12$ and $\Omega_0(x)=\pi^{-1/4}e^{-(1/2)x^2}$ so that $L_0\Omega_0=0$ and $\int |\Omega_0|^2=1$. Let moreover $f_0,\dots,f_n\in L^\infty(\R)$ and let $-\infty<y_0<\dots<y_n<+\infty$. Then
$$
\E(f_0(q(y_0)),\dots, f_n(q(y_n))=(\Omega, M_{f_0}e^{-x_1L_0}M_{f_1}\dots e^{-x_nL_0}M_{f_n}\Omega_0)_{L^2}
$$
where $x_i=y_i-y_{i-1}>0$, $(\cdot,\cdot)_{L^2}$ denotes the standard $L^2$ scalar product and $M_f$ the multiplication operator $M_fg(x)=f(x)g(x)$.
\end{proposition}
\begin{proof}
See \cite{SIMfunint} Theorem 4.7 pag. 37.
\end{proof}
It is also possible to give an analogous result in a slightly more general setting, i.e. to relate the semigroup $e^{-xL}$ with $L=L_0+V$ for some suitable potential $V$ to path integrals. Results of this kind have been widely investigated in literature, especially in the case of Brownian motion, and are usually referred to as {\em Feynman-Kac} formulas. In what follows $V$ will be such that $E(V)=\inf {\rm spec} (L_0+V)$ is a simple eigenvalue with an associated strictly positive eigenvector $\Omega_V$; polynomials bounded from below satisfy this property.
\begin{definition}
We define the $P(\phi)_1$-process the stochastic process with joint distribution $q(x_1),\dots q(x_n)$, $(x_1<\dots<x_n)$:
$$
\Omega_V(u_1)\Omega_V(u_n) e^{-y_1\hat L}(u_1,u_2)\dots e^{-y_{n-1}}\hat{L}(u_{n-1},u_n).
$$ and $y_i=x_{i+1}-x_i$
where $e^{-y\hat{L}}(a,b)$ is the integral kernel of $e^{-s\hat{L}}$. We will denote with $d\rho_V$ the corresponding measure.
\end{definition}
We get the following result:
\begin{theorem}\label{feykac}[Feynman-Kac]
For any smooth and bounded test function $G: C(\R,\C)\rightarrow \C$ we have
$$
\int G(u) d\rho_{3,L}=\displaystyle \lim_{R\rightarrow+\infty} C_R^{-1}\int G(u) e^{\int_{-R}^RV(|u(x)|^2)}d\mu(u)
$$
where $C_N$ is the $L^1(d\mu)$ norm of $e^{-\int_{-R}^R V(|u(x)|^2}d\mu(u)$.
\end{theorem}
\begin{proof}
See\cite{SIMfunint} Theorems 6.7 and 6.9 pag 58.  Notice that, by mimicking the proof of Theorem
6.1 there, it is possible to deal also with the complex case.
\end{proof}

We can also describe $\rho_{L,3}$ thanks to Feynman-Kac formulas by the following theorem.

\begin{theorem}\label{th-describrhol3}[\cite{SIMfunint}, Theorem 6.7] We have that the restriction to the $\sigma$-algebra generated by $\sigma_{a,b}:=\{q(x)\}_{a\leq x\leq b}$ of $\mu_V$ is absolutely continuous with regard to $\mu$ restricted to the same $\sigma$-algebra with 
$$
\frac{d\mu_V|_{\sigma_{a,b}}}{d\mu|_{\sigma_{a,b}}} (u) = \Omega_V(u(b))\Omega_V(u(a))\Omega_0^{-1}(u(b)) \Omega_0^{-1}(u(a)) e^{E(V)(b-a)}e^{-\int_{a}^b V(|u|^2(x)) dx}.
$$
\end{theorem}

\paragraph{Some probabilistic estimates.}
\begin{proposition}\label{prop-prob1} We have that for all $p\geq 2$, and $s< \frac12$, there exists $C$ such that for all $x\in \R$
$$
\|D^s(\xi - \xi_L)(x)\|_{L^p_{\textrm{proba}}} \leq CL^{-1} \an x.
$$
\end{proposition}

\begin{remark}\label{rkproba}
Notice that the space $L^p_{\textrm{proba}}$ is short for the $L^p$ space of the probabilistic space where the Gaussian process $W$ is defined.
\end{remark}

This is due to the fact that $D^s(\xi - \xi_L)(x)$ is a Gaussian variable hence 
$$
\|D^s(\xi - \xi_L)(x)\|_{L^p_{\textrm{proba}}} \lesssim \|D^s(\xi - \xi_L)(x)\|_{L^2_{\textrm{proba}}} .
$$

What is more,
$$
D^s(\xi - \xi_L)(x) = \int \Big( \frac{e^{ikx}}{(1+k^2)^{1/2-s}} - \frac{e^{i\ent{k}_L x}}{(1+\ent{k}_L^2)^{1/2-s}}\Big) dW(k)
$$
where $\ent{k}_L = L^{-1} \ent{kL} =\frac1L \max \{n\in\Z:n\leq kL \}$. Since 
$$
\Big| \frac{e^{ikx}}{(1+k^2)^{1/2-s}} - \frac{e^{i\ent{k}_L x}}{(1+\ent{k}_L^2)^{1/2-s}}\Big| \lesssim \frac{\an x}{L} (1+k^2)^{s-1/2}
$$
we get the result.

\begin{proposition}\label{prop-prob2} From Feynman-Kac's theory, we have that for all $r\geq 2, s< \frac12$, such that $r(\frac12 - s) \geq \frac72$ (that is for further reference $r \geq r_s := \frac7{1-2s}$), there exists $\varphi_r$ $\varphi_{r,s}$ such that for all $x,y\in \R$, $L\geq 1$,
$$
\int |u(x)|^r d\rho_{L,3} (u) \leq \varphi_{r}(|x|)
$$
and
$$
\int \frac{|u(x)-u(y)|^r}{|x-y|^{1+rs}} d\rho_{L,3}(u) \leq \varphi_{r,s}(\max(|x|,|y|)).
$$
The map $\varphi_r$ is given by
$$
\varphi_r(x) = C_r \max_{R(L) \leq |x|} L^{1/6}
$$
where $C_r$ is a constant depending on $r$.
\end{proposition}

\begin{proof} 

For the second inequality, we assume that $|x-y|\leq 1$ and we use the description of the measure and the results seen in the previous paragraph. We consider the operator $T_V$ defined as 
$$T_Vf(u) = -\lap f(u)+ (|u|^2 + V(|u|^2) -\frac12)f(u),$$ and let $\Omega_V$ be the eigenstate associated to the non-degenerate first eigenvalue $E(V)$ of $T_V$. We also define the operator $\hat{T}_V=T_V-E(V)$.

Let $x,y \in \R$ and let $R(L) \geq \max (|x|,|y|)$. We assume, without loss of generality, $x\geq y$. 
We now rely on Theorem \ref{th-describrhol3} on the map $u\mapsto G(u)$ defined as:
$$
G(u) = \frac{|u(x)-u(y)|^r}{|x-y|^{1+s r}}\Omega_0(u(-R(L))) \Omega_V^{-1}(u(-R(L)))\Omega_0(u(R(L))) \Omega_V^{-1}(u(R(L)))e^{-2E(V) R(L)}Z_{L,3}^{-1}.
$$
We get on one hand that
$$
\int G(u) \Omega_V(u(-R(L))) \Omega_0^{-1}(u(-R(L)))\Omega_V(u(R(L))) \Omega_0^{-1}(u(R(L)))e^{2E(V) R(L)} Z_{L,3} d\rho_{L,3}(u)
$$
is equal to 
$$
\int \frac{|u(x)-u(y)|^r}{|x-y|^{s r+1}} d\rho_{L,3}(u),
 $$
and on the other hand that is equal to 
\begin{multline*}
\int \frac{|u_x-u_y|^r}{|x-y|^{s r+1}} \Omega_0(u_{-R(L)}) \Omega_V^{-1}(u_{-R(L)})\Omega_0(u_R(L)) \Omega_V^{-1}(u_R(L))e^{-2E(V) R(L)}Z_{L,3}^{-1} \Omega_V(u_R(L)) \Omega_V(u_{-R(L)})\\
e^{-(y+R(L))\hat T_V} (u_{-R(L)},u_y) e^{-(x-y)\hat T_V} (u_y,u_x) e^{-(R(L)-x)\hat T_V} (u_x,u_R(L)) du_{-R(L)}du_ydu_x du_R(L).
\end{multline*}
By simplifying the $\Omega_V$ we get
\begin{multline*}
\int \frac{|u(x)-u(y)|^r}{|x-y|^{s r +1}} d\rho_{L,3}(u)= \\
 e^{-2R(L)E(V)}Z_{L,3}^{-1}\int  \tilde G
e^{-(y+R(L))\hat T_V} (u_{-R(L)},u_y) e^{-(x-y)\hat T_V} (u_y,u_x) e^{-(R(L)-x)\hat T_V} (u_x,u_R(L)) du_{-R(L)}du_ydu_x du_R(L)
\end{multline*}
with
$$
\tilde G(u_{-R(L)},u_y,u_x,u_R(L)) =
\frac{|u_x-u_y|^r}{|x-y|^{s r +1}} \Omega_0(u_{-R(L)}) \Omega_0(u_R(L))  .
$$
By applying Theorem \ref{th-describrhol3} on the map $u\mapsto \Omega_0(u_{-R(L)}) \Omega_V^{-1}(u_{-R(L)})\Omega_0(u_R(L)) \Omega_V^{-1}(u_R(L))e^{-2E(V) R(L)}$, we get that
$$
Z_L = e^{-2E(V) R(L)}\int \Omega_0(u_1) \Omega_0(u_2) e^{-2R(L)\hat T_V}(u_1,u_2) du_1du_2 = e^{-2E(V) R(L)}\an{\Omega_0,e^{-2R(L)\hat T_V}\Omega_0}_{L^2}.
$$
Indeed, 
$$
e^{-2R(L)\hat T_V}\Omega_0 (u) = \int dv e^{-2R(L)\hat T_V}(v,u) \Omega_0(v).
$$
By decomposing $\Omega_0$ on $\R \Omega_V$ and its orthogonal, we get
$$
Z_L \geq e^{-2E(V) R(L)} \an{\Omega_0,\Omega_V}^2.
$$
Because $\Omega_0$ and $\Omega_V$ are both positive maps, we get that $c^{-1}:=\an{\Omega_0,\Omega_V}^2>0$.

Using the maximum principle, we get
\begin{multline*}
\int \frac{|u(x)-u(y)|^r}{|x-y|^{s r +1}} d\rho_{L,3}(u)\leq c
\int \frac{|u_x-u_y|^r}{|x-y|^{s r +1}} \Omega_0(u_{-R(L)}) \Omega_0(u_R(L))   e^{-(y+R(L))\hat T_V} (u_{-R(L)},u_y) \\
e^{-(x-y)(\hat T_0 - E(V)) } (u_y,u_x) e^{-(R(L)-x)\hat T_V} (u_x,u_R(L)) du_{-R(L)}du_ydu_x du_R(L).
\end{multline*}
Using Mehler's formula (see \cite{SIMfunint} pag. 38), we get that for $r(\frac12 -s) \geq \frac72$,  there exists $C_{r,s}$
$$
e^{-(x-y)\hat T_0} (u_y,u_x)\frac{|u_x-u_y|^r}{|x-y|^{s r +1}} \leq C_{r,s} \frac1{1+ |u_x|^3+ |u_y|^3}.
$$
Indeed, applying Mehler's formula to our case, we need to bound the following term
$$
\displaystyle
\left|\frac{u_x-u_y}{|x-y|^{1/2}}\right|^re^{-\frac{|u_x-u_y|^2}{|x-y|}} |x-y|^{r(\frac12-s)-2}e^{-\frac{u_x^2+u_y^2}2|x-y|}:
$$
the first term $\left|\frac{u_x-u_y}{|x-y|^{1/2}}\right|^re^{-\frac{|u_x-u_y|^2}{|x-y|}} $ is bounded, while for the second one we have, as $r(\frac12-s)-2\geq \frac32$, $$|x-y|^{r(\frac12-s)-2}e^{-\frac{u_x^2+u_y^2}2|x-y|}\lesssim \frac1{1+ |u_x|^3+ |u_y|^3}.$$
Therefore 
\begin{multline*}
\int \frac{|u(x)-u(y)|^r}{|x-y|^{s r +1}} d\rho_{L,3}(u)\leq c C_{r,s}
\int  \Omega_0(u_{-R(L)}) \Omega_0(u_R(L))   e^{-(y+R(L))\hat T_V} (u_{-R(L)},u_y) \\
e^{(x-y) E(V) } \frac1{\an{u_x}^{3/2}} \frac1{\an{u_y}^{3/2}} e^{-(R(L)-x)\hat T_V} (u_x,u_R(L)) du_{-R(L)}du_ydu_x du_R(L).
\end{multline*}

Integrating over $u_{-R(L)}$ and $u_x$ yields 
$$
\int \frac{|u(x)-u(y)|^r}{|x-y|^{s r +1}}d\rho_{L,3}(u)\leq  \an{\Omega_0,  e^{-(R(L)-x)\hat T_V} \frac1{\an{u_x}^{3/2}} } \an{e^{-(y+R(L))\hat T_V}\Omega_0,\frac1{\an{u_y}^{3/2}}}e^{(x-y) E(V)}.
$$
Since $e^{-(y+R(L))\hat T_V}$ and $e^{-(R(L)-x)\hat T_V}$ are less than the identity and $u\mapsto \frac1{\an{u}^{3/2}}$ and $\Omega_0$ belong to $L^2$ (because $u$ is complex, we are in $\R$ dimension 2), we get that
$$
\int \frac{|u(x)-u(y)|^r}{|x-y|^{s r +1}}d\rho_{L,3}(u)\leq C'_{r,s}.
$$

For $R(L)\leq \max(|x|,|y|)$, we have
$$
\int \frac{|u(x)-u(y)|^r}{|x-y|^{s r +1}} d\rho_{L,3}(u) \leq \frac1{Z_{L,3}} \int \frac{|u(x)-u(y)|^r}{|x-y|^{s r +1}}d\mu(u)
$$
where we have thanks to Assumption \eqref{assum-R}, $\frac1{Z_{L,3}} \leq L^{1/6}$. Thus, with
$$
\varphi_{r} (|x|) = C_{r,s}\max_{R(L)\leq |x|} C L^{1/6},
$$
we get the second inequality for $|x-y| \leq 1$.

We have that $\varphi_{r,} $ is a non negative, increasing function.

The first estimate was proved in \cite{bouinf} in the case of $ u \mapsto V(|u|^2)$ convex but can be proved in the same way as the second in the general case. From Theorem \ref{th-describrhol3}, we get for $x \in [-R(L),R(L)]$,
\begin{multline*}
\int |u(x)|^r d\rho_{L,3}(u) =\\ \frac{e^{-2R(L)E(V)}}{Z_L} \int du_{R}du_{-R} du_x \Omega_0(u_R) \Omega_0(u_{-R}) e^{-\hat T_V(R(L) -x)}(u_x,u_R) e^{-\hat T_V (x+R(L))}(u_{-R},u_x)|u_x|^r.
\end{multline*}
If $x+R(L) \geq 1$, we decompose the integral as
\begin{multline*}
\int |u(x)|^r d\rho_{L,3}(u) =\\ \frac{e^{-2R(L)E(V)}}{Z_L} \int du_{R}du_{-R} du_x du_y \Omega_0(u_R) \Omega_0(u_{-R}) e^{-\hat T_V(R(L) -x)}(u_x,u_R) e^{-\hat T_V }(u_y,u_x)e^{-\hat T_V (y+R(L))}(u_{-R},u_y)|u_x|^r
\end{multline*}
where $y = x-1 \in [-R(L),R(L)]$. Otherwise we decompose it as 
\begin{multline*}
\int |u(x)|^r d\rho_{L,3}(u) =\\ \frac{e^{-2R(L)E(V)}}{Z_L} \int du_{R}du_{-R} du_x du_y \Omega_0(u_R) \Omega_0(u_{-R}) e^{-\hat T_V(R(L) -y)}(u_y,u_R) e^{-\hat T_V }(u_x,u_y)e^{-\hat T_V (x+R(L))}(u_{-R},u_x)|u_x|^r
\end{multline*}
where $y=x+1 \in [-R(L),R(L)]$ taking $L$ large enough to have $R(L)\geq 1$. We focus on the first case. Thanks to Mehler's formula and the maximum principle, we have 
$$
e^{-\hat T_V }(u_y,u_x) |u_x|^r \leq C_r e^{- (|u_x|^2 + |u_y|^2)/4} e^{E(V)}.
$$
Therefore, we have 
\begin{multline*}
\int |u(x)|^r d\rho_{L,3}(u)  \leq \\ c C_r e^{E(V)}\int du_{R}du_{-R} du_x du_y \Omega_0(u_R) \Omega_0(u_{-R}) e^{-\hat T_V(R(L) -x)}(u_x,u_R) e^{-\hat T_V (y+R(L))}(u_{-R},u_y)e^{-(|u_x|^2+ |u_y|^2)/4}.
\end{multline*}
We integrate over $u_{-R}$ and $u_x$ to get
\begin{multline*}
\int |u(x)|^r d\rho_{L,3}(u)  \leq \\ c C_r e^{E(V)}\int du_{R}  du_y \Omega_0(u_R)  e^{-\hat T_V(R(L) -x)}(e^{-|z|^2/4})(u_R) e^{-\hat T_V (y+R(L))}(\Omega_0)(u_y)e^{- |u_y|^2/4},
\end{multline*}
that is
$$
\int |u(x)|^r d\rho_{L,3}(u) \leq c C_r e^{E(V)} \an{\Omega_0,e^{-\hat T_V(R(L) -x)}(e^{-|z|^2/4})} \an{e^{-\hat T_V (y+R(L))}(\Omega_0), e^{-|z|^2/4}},
$$
from which we deduce since $\hat T_V \geq 0$ and $\Omega_0$ and $e^{-|z|^2/4}$ are in $L^2$
$$
\int |u(x)|^r d\rho_{L,3}(u) \leq C'_r
$$
where $C'_r$ does not depend on $x$ or $L$. We deduce the case $x\notin [-R(L),R(L)]$ as for the second inequality. 

For $|x-y|\geq 1$ in the second inequality, we estimate $\big| \frac{u(x)-u(y)}{x-y}\big|^r$ by $|u(x)|^r+|u(y)|^r$ and we use the first inequality.

%
%

\end{proof}

\section{The Prokhorov-Skorohod method and the reduction to rough estimates and convergence}

\subsection{The Prokhorov-Skorohod method}

We start by recalling Prokhorov's and Skorohod's Theorems.

\begin{theorem}[Prokhorov]\label{th-prok} Let $(\nu_L)_L$ be a family of probability measures defined on the topological $\sigma$ algebra of a separable complete metric space $S$. Assume that $(\nu_L)_L$ is tight, that is, for all $\varepsilon > 0$, there exists a compact $K_\varepsilon$ of $S$ such that for all $L$, we have 
$$
\nu_L(K_\varepsilon) \geq 1-\varepsilon.
$$
Then there exists a sequence $L_n$ such that $\nu_{L_n}$ converges weakly. That is, there exists a probability measure on $S$, $\nu$ such that for all functions $F$ bounded and Lipschitz continuous on $S$, we have 
$$
\E_{\nu_{L_n}}(F) \rightarrow \E_{\nu}(F).
$$
\end{theorem}

We refer to \cite{prok}, page 114.

\begin{theorem}[Skorohod]\label{th-sko} Let $\nu_n$ be sequence of probability measures defined on the topological $\sigma$ algebra of a separable complete metric space $S$. Assume that $(\nu_n)_n$ converges weakly towards a probability measure $\nu$. Then there exists a subsequence $\nu_{n_k}$ of $(\nu_n)_n$, a probability space $(\Omega, \mathcal A, P)$, a sequence of random variable on this space $(X_k)_k$ and a random variable $X_\infty$ on this space such that \begin{itemize}
\item for all $k$, the law of $X_k$ is $\nu_{n_k}$, that is for all measurable set $A$ of $S$, $\nu_{n_k}(A) = P(X_k^{-1}(A))$,
\item the law of $X_\infty$ is $\nu$,
\item the sequence $X_k$ converges almost surely in $S$ towards $X_\infty$.
\end{itemize}
\end{theorem}

We refer to \cite{skoro}, page 79.

From the combination of these two Theorems, we get the following

\begin{corollary}\label{cor-PSmeth}
Let $S$ and $S_s$ be as in \eqref{defSs}-\eqref{defS}, and let $(\nu_L)_L$ be a family of probability measures defined on the topological $\sigma$ algebra of $S$.
For all $R\geq 0$, let $B_R$ be the closed ball of $S_s$ of center $0$ and radius $R$. Assume that
\begin{itemize}
\item for all $R\geq 0$, the ball $B_R$ is compact in $S$,
\item there exists $C\geq 0$ such that for all $L$, we have 
$$
\int \|u\|_{S_s}^2 d\nu_L(u)\leq C.
$$
\end{itemize}
Then, there exists a sequence $L_n$, a probability space $(\Omega, \mathcal A, P)$, a sequence of random variable on this space $(X_n)_n$ and a random variable $X_\infty$ on this space such that \begin{itemize}
\item for all $n$, the law of $X_n$ is $\nu_{L_n}$, 
\item the sequence $X_n$ converges almost surely in $S$ towards $X_\infty$.
\end{itemize}
\end{corollary}

\begin{proof} The proof uses Markov's inequality : 
$$
\nu_L( \|u\|_{S_s} > R) \leq R^{-2} C
$$
therefore
$$
\nu_L ( B_{R_\varepsilon}) \geq 1- \varepsilon
$$
for $C R_\varepsilon^{-2} \leq \varepsilon$. And $B_{R_\varepsilon}$ is compact in $S$. Then, one can apply Prokhorov's theorem and then Skorohod's theorem to conclude.
\end{proof}

We justify our choice for $S_s$. From now on, $S_s$ and $S$ are the spaces defined in the first section \eqref{defSs}, \eqref{defS}.

\begin{proposition}\label{compball} Let $B_R$ be the ball of $S_s$ of center $0$ and radius $R$. For all $R\geq 0$, $B_R$ is compact in $S$. \end{proposition}

\begin{proof} The proof is classical so we keep it short. Let $\eta $ be a $\mathcal C^\infty(\R_+)$ function with compact support. Assume that $\eta$ is such that $\eta(r) = 1$ if $r\leq 1$, $\eta(r) = 0$ if $r\geq 2$.

Let $f \in B_R$ and let $\varepsilon > 0$.

Let $f^T = \eta(|t|/T) f$. We have thanks to Sobolev's inequality on the time norm,
$$
\|f-f^T\|_S \leq C \an{T}^{-1/2} \|f\|_{S_s}
$$
where $C$ is a universal constant. Thus,
$$
\|f-f^T\|_S \leq C \an{T}^{-1/2}R.
$$
We choose $T$ such that $ C \an{T}^{-1/2}R \leq \frac\varepsilon{5}$.

Let $f^{T,F} = \eta\Big( \frac{1-\partial_t^2}{F^2}\Big) f^T$. We have, thanks to Sobolev's inequality on the time norm
$$
\|f^{T,F}-f^T\|_{S} \leq C(T)\|(1-\partial_t^2)^{3/8}(f^{T,F}-f^T) \|_{L^2(\R,H_\varphi)} 
$$
and thus
$$
\|f^{T,F}-f^T\|_{S} \leq C(T)F^{-1/4}\|(1-\partial_t^2)^{1/2}f^T \|_{L^2(\R,H_\varphi)} \leq C(T) F^{-1/4} R
$$
where $C(T)$ is a constant depending only on $T$. We choose $F$ such that $ C(T) F^{-1/4} R\leq  \frac\varepsilon{5}$.

Let $f^{T,F,X}$ be $\eta \Big( \frac{|x|}{X}\Big)f^{T,F}$. We have 
$$
\|f^{T,F,X}-f^{T,F}\|_{S} \leq C(T,F)X^{-1}\|f^{T,F} \|_{S_s} \leq C(T,F) X^{-1} R
$$
where $C(T,F)$ is a constant depending only on $T$ and $F$. We choose $X$ such that $ C(T,F) X^{-1} R\leq  \frac\varepsilon{5}$.

Let $f^{T,F,X,N} = \eta\Big( \frac{1-\partial_x^2}{N^2}\Big)f^{T,F,X}$, we have 
$$
\|f^{T,F,X,N}-f^{T,F,X}\|_{S} \leq C(T,F,X)N^{-s}\|f^{T,F,X} \|_{S_s} \leq C(T,F,X) N^{-s} R
$$
where $C(T,F,X)$ is a constant depending only on $T,X$ and $F$. We choose $N$ such that 
$$
C(T,F,X) N^{-s} R\leq  \frac\varepsilon{5}.
$$

Finally, we have that 
$$
\|f^{T,F,X,N}\|_S \leq C(T,F,X,N) R
$$
where $C(T,F,X,N)$ is a constant depending only on $T,F,X,N$. 

What is more, $ f^{T,F,X,N}$ as a function on $[-2T,2T]\times [-2X,2X]$, belongs to
$$
\textrm{Vect }\Big( \Big\{ (t,x)\mapsto e^{i(\omega t + kx)} \Big| \omega \in \frac\pi{2T}\Z \cap [-F,F] \, ,\, k\in \frac{\pi}{2X}\Z \cap [-N,N]\Big\}\Big)
$$
which is of finite dimension.

Hence, there exists a finite family of function $f_1,\hdots, f_{N_\varepsilon}$ of $S$ such that for all $f \in B_R$,
$$
f^{T,F,X,N} \in \bigcup_{k=1}^{N_\varepsilon} B^S\Big(f_k, \frac{\varepsilon}{5}\Big)
$$
where $B^S(f_k, \frac{\varepsilon}{5})$ is the open ball of $S$ of center $f_k$ and radius $\frac{\varepsilon}{5}$. Therefore, for all $f\in B_R$,
$$
f \in \bigcup_{k=1}^{N_\varepsilon} B^S(f_k, \varepsilon).
$$

Therefore, $B_R$ is totally bounded in $S$. Since $S$ is a complete normed space, all we have to prove is that $B_R$ is closed in $S$. 

Let $(\eta_M)_{M\in \N} $ be a sequence of smooth maps with compact support in $\R^2$ equal to $1$ on $[-M,M]^2$ and such that $\grad_{t,x} \eta_M$ goes to $0$ in $L^\infty$ as $m$ goes to $\infty$. Let $\Pi_N$ be the Fourier multiplier by $\textbf{1}_{|\xi|\leq N}$ in both time and space.

For any $f \in S$, $f$ belongs to $S_s$ if
$$
 \sup_{M\in \N} \sup_{N\in \N}  \Big( \|\Pi_N \eta_M \an{t}^{-1}\varphi D^{s-\kappa} f\|_{L^2}^2 + \|\Pi_N \eta_M \an{t}^{-1}\varphi D^{s-\kappa} \partial_t f\|_{L^2}^2 \Big)
$$
is finite and in this case it is equal to $\|f\|_{S_s}^2$.

For any $f\in S$, $N,M \in \N^2$
$$
K_{M,N}(f)^2 =  \|\Pi_N \eta_M \an{t}^{-1}\varphi D^{s-\kappa} f\|_{L^2}^2 + \|\Pi_N \eta_M \an{t}^{-1}\varphi D^{s-\kappa} \partial_t f\|_{L^2}^2 .
$$
We have 
$$
K_{M,N}(f) \leq N^s \an{M}^{9/2}\varphi(M)^{-1} \|D_t D_x^s \eta_M \an{t}^{-1} \varphi\|_{L^\infty} \|f\|_{S}.
$$

Let $(f_n)_n$ a sequence of $B_R$ that converge to $f$ in $S$. Let $\varepsilon > 0$, and $N,M\in \N^2$. We have by triangular inequality for all $n\in \N$,
$$
K_{M,N(f)} \leq K_{M,N}(f_n) + K_{M,N}(f-f_n).
$$
Because $f_n \in B_R$, we have $K_{M,N}(f_n) \leq R$. Since $f_n$ converges towards $f$ in $S$ we have that for $n$ big enough
$$
K_{M,N}(f-f_n) \leq \varepsilon.
$$
Therefore,
$$
K_{M,N}(f) \leq R + \varepsilon.
$$
By taking the supremum in $M$ and $N$, we get that $f\in S_s$ and
$$
\|f\|_{S_s} \leq R+ \varepsilon
$$
and we let $\varepsilon$ go to $0$ to get the result.

\end{proof}

\subsection{Reduction to rough estimates and convergence}

\begin{proposition}\label{prop1} Assume that for all $x$, and with $\varphi_r(x)$ the map given by Proposition \ref{prop-prob2}, there exists $C_r>0$ such that
$$
\int |u(x)|^r d\rho_L(u) \leq C_r \varphi_r (|x|).
$$
Then, there exists a positive, even, decreasing on $\R^+$  map $\varphi$ such that the Prokhorov-Skorohod method applies, that is,  there exists a sequence $L_n$, a probability space $(\Omega, \mathcal A, P)$, a sequence of random variables on this space $(X_n)_n$ and a random variable $X_\infty$ on this space such that \begin{itemize}
\item for all $n$, the law of $X_n$ is $\nu_{L_n}$, 
\item the sequence $X_n$ converges almost surely in $S$ towards $X_\infty$.
\end{itemize}
\end{proposition}

\begin{remark}
Notice that the probability space introduced in proposition above is different from the one underlying the construction of the measure $\rho_L$; we will use the notation $L^p(\Omega)$ to denote Lebesgue spaces with respect to this measure, as opposed to the $L^p_{\rm proba}$ of Remark \ref{rkproba}.
\end{remark}

\begin{proof} Given Corollary \ref{cor-PSmeth}, all we have to do is prove that there exists $C\geq 0$ such that for all $L$, we have 
$$
\int \|u\|_{S_s}^2 d\nu_L(u)\leq C.
$$

Take $\varphi$ even decreasing on $\R^+$ and such that
$$
\sum_n \varphi(n)^2 \int_{n}^{n+1} \varphi_2(x)dx 
$$
converges.

We have 
$$
\int \|u\|_{S_s}^2 d\nu_L(u) =A +B
$$
with 
$$
A  = \int_{S}\int_{\R}dt \an{t}^{-2} \|\varphi D^{s-\kappa} u(t) \|_{L^2(\R)}^2 d\nu_{L}(u)
$$
and 
$$
B = \int_{S}\int_{\R}dt \an{t}^{-2} \|\varphi D^{s-\kappa} \partial_t u(t) \|_{L^2(\R)}^2 d\nu_L(u).
$$
We use the definition of $\nu_L$ in terms of the flow $\psi_L$ to get
$$
A  = \int_{H_\varphi}\int_{\R}dt \an{t}^{-2} \|\varphi D^{s-\kappa}\psi_L(t) u\|_{L^2(\R)}^2 d\rho_{L}(u)
$$
and 
$$
B = \int_{H_\varphi}\int_{\R}dt \an{t}^{-2} \|\varphi D^{s-\kappa} \partial_t \psi_L(t) u \|_{L^2(\R)}^2 d\rho_L(u).
$$
We can exchange the integral in time and in probability to get
$$
A  =\int_{\R}dt \an{t}^{-2} \int_{H_\varphi} \|\varphi D^{s-\kappa}\psi_L(t) u \|_{L^2(\R)}^2 d\rho_{L}(u)
$$
and 
$$
B =\int_{\R}dt \an{t}^{-2} \int_{H_\varphi} \|\varphi D^{s-\kappa} \partial_t \psi_L(t)u \|_{L^2(\R)}^2 d\rho_L(u).
$$
We use the fact that $\psi_L(t)u$ solves the equation
$$
i\partial_t \psi_L(t) u = -\Pi_{N(L)}J_L \Pi_{N(L)}\lap \psi_L(t) u + \Pi_{N(L)}J_L\Pi_{N(L)} \chi_L V'(|\psi_L(t)u|^2)\psi_L(t) u 
$$
to get
\begin{multline*}
B  \leq \int_{\R}dt \an{t}^{-2} \int_{H_\varphi} \|\varphi D^{s-\kappa} \Pi_{N(L)}J_L \Pi_{N(L)}\lap \psi_L(t) u \|_{L^2(\R)}^2 d\rho_L(u) + \\
\int_{\R}dt \an{t}^{-2} \int_{H_\varphi} \|\varphi D^{s-\kappa} \Pi_{N(L)}J_L \Pi_{N(L)} \chi_L V'(|\psi_L(t)u|^2)\psi_L(t) u \|_{L^2(\R)}^2 d\rho_L(u).
\end{multline*}
Using the invariance of $\rho_L$ under $\Psi_L$ we get
\begin{multline*}
B  \leq \int_{\R}dt \an{t}^{-2} \int_{H_\varphi} \|\varphi D^{s-\kappa} \Pi_{N(L)}J_L \Pi_{N(L)}\lap u \|_{L^2(\R)}^2 d\rho_L(u) + \\
\int_{\R}dt \an{t}^{-2} \int_{H_\varphi} \|\varphi D^{s-\kappa} \Pi_{N(L)}J_L \Pi_{N(L)} \chi_L V'(|\psi_L(t)u|^2) u \|_{L^2(\R)}^2 d\rho_L(u).
\end{multline*}

Using the decreasing character of $\varphi$, we get
$$
\int \|\varphi D^{s-\kappa} \Pi_{N(L)}J_L \Pi_{N(L)}\lap  u \|_{L^2(\R)}^2 d\rho_L (u) \leq I + II
$$
with 
$$
I = \sum_{n\in \N } \varphi(n)^2 \int \| D^{s-\kappa} \Pi_{N(L)}J_L \Pi_{N(L)} \lap  u  \|_{L^2([n,n+1[}^2 d\rho_L(u)
$$
and 
$$
II = \sum_{n\in -\N } \varphi(n)^2 \int \| D^{s-\kappa}\Pi_{N(L)} J_L \Pi_{N(L)} \lap  u  \|_{L^2([n,n+1[}^2 d\rho_L(u).
$$
We have 
$$
 \| D^{s-\kappa} \Pi_{N(L)}J_L \Pi_{N(L)} \lap  u  \|_{L^2([n,n+1)[}^2  \lesssim \alpha + \beta
$$
With 
$$
\alpha = \| D^{s-\kappa} J_L \Pi_{N(L)} \lap  u  \|_{L^2([n,n+1)[}^2
$$
and 
$$
\beta = \| D^{s-\kappa}(1- \Pi_{N(L)})J_L \Pi_{N(L)} \lap  u  \|_{L^2([n,n+1])}^2.
$$
We have that 
$$
D^{s-\kappa} J_L \lap = D^{s-\kappa} \eta_L D^{\kappa - s} D^{s-\kappa} J (\lap - 1) (\lap - 1)^{-1} \eta_L (\lap - 1) + D^{s-\kappa} \eta_L D^{\kappa - s} D^{s-\kappa} J \eta_L.
$$
Because $\eta_L$ is smooth, bounded and with derivatives bounded, we get that 
$$
\|D^{s-\kappa} \eta_L D^{\kappa - s} \|_{L^2\rightarrow L^2}  \textrm{ and } \|(\lap - 1)^{-1} \eta_L (\lap - 1) \|_{L^2\rightarrow L^2}
$$
are bounded uniformly in $L$. Therefore, using the locality of $J$ and Assumption \ref{assum-J},
$$
\alpha \lesssim \|\Pi_N u \|_{L^2([n,n+1)[}^2.
$$
We use that $\rho_L$ almost surely $\Pi_Nu = u$, to get
$$
\int \alpha d\rho_L(u)  = \int  \| u\|_{L^2([n,n+1])}^2 d\rho_L(u) \leq \int_{n}^{n+1} \varphi_2(|x|)
$$
For $\beta$ we use that
$$
\beta \leq \| D^{s-\kappa}(1- \Pi_{N(L)})J_L \Pi_{N(L)} \lap  u  \|_{L^2([n, n + 2\pi L])}^2.
$$
Because of $2\pi L$ periodicity, we get
$$
\beta \leq N(L)^{-1/4} \| D^{s-\kappa+1/4}J_L \Pi_{N(L)} \lap  u  \|_{L^2([0,  2\pi L])}^2.
$$
Because of the assumption on $J$ we get
$$
\beta \leq N(L)^{-1/4} \| u\|_{L^2([0,2\pi L])}^2
$$
and thus
$$
\int \beta d\rho_L(u) \leq N(L)^{-1/4} \int \varphi_2(|x|).
$$
We use the assumptions on $N(L)$ to get that the right hand side above is uniformly bounded in $L$.

We get 
$$
 I  \lesssim \sum_{n \in \N} \varphi(n)^2 \left(1+ \int_{n}^{n+1} \varphi_2(x) dx\right) < \infty
$$

We proceed in the same way for $II$ and we integrate to get 
$$
\int_{\R}dt \an{t}^{-2} \int_{H_\varphi} \|\varphi D^{s-\kappa} \Pi_{N(L)}J_L \Pi_{N(L)}\lap \psi_L(t) u \|_{L^2(\R)}^2 d\rho_L(u)  \lesssim
\int_{\R}dt \an{t}^{-2} 
$$
and since $\an{t}^{-2}$ is integrable, we get a first estimate.

We proceed in the same way for the second part of $B$ and we get
\begin{multline*}
\int_{H_\varphi} \|\varphi D^{s-\kappa} \Pi_{N(L)}J_L \Pi_{N(L)} \chi_L V'(|\psi_L(t)u|^2) \psi_L(t) u \|_{L^2(\R)}^2 d\rho_L(u) \lesssim \\ \int_{H_\varphi} \|\an u^{r_V+1}\|_{L^2([n,n+1]}^2 d\rho_L(u)+N(L)^{-1/4} \int_{H_\varphi} \|\an u^{r_V+1}\|_{L^2([0,2\pi L]}^2 d\rho_L(u).
\end{multline*}
We get
$$
\int_{H_\varphi} \|\varphi D^{s-\kappa} \Pi_{N(L)}J_L \Pi_{N(L)} \chi_L V'(|\psi_L(t)u|^2) \psi_L(t) u \|_{L^2(\R)}^2 d\rho_L(u) \lesssim   \int_{n}^{n+1}\varphi_{2r_V+2}dx+N(L)^{-1/4} \int_{0}^{2\pi L} \varphi_{2r_V+2}.
$$
We recall $\varphi_r$ and $\varphi_2$ differ only by a constant independent from $L$ to get that $B$ is bounded uniformly in $L$.

For A, we use the invariance of $\rho_L$ under $\psi_L(t)$ to get
$$
A = \int_{\R}dt \an{t}^{-2} \int_{H_\varphi} \|\varphi D^{s-\kappa} u \|_{L^2(\R)}^2 d\rho_{L}(u)
$$
and we proceed in the same way to get that $A$ is bounded independently from $L$.

\end{proof}

\begin{proposition}\label{prop2} Assume that $\rho_L \rightarrow \rho$ weakly in $H_\varphi$. Assume that for all $r\geq 2$, $s<\frac12$, there exist $C_{r,s}$ and a positive, even, decreasing on $\R^+$ map $\varphi_1$ such that for all $L$
$$
\int \Big( \| \varphi_1 D^s u \|_{L^2(\R)}^r \Big) d\rho_L(u)\leq C_{r,s}.
$$
Then, the random variable $X_\infty$ given by the Prokhorov-Skorohod method satisfies \begin{itemize}
\item for all $t\in \R$, the law of $X_\infty$ is the weak limit $\rho$ of $\rho_{L_n}$, and thus do not depend on time,
\item $X_\infty$ is a weak solution (in the sense of distribution) of 
$$
\partial_t u = -J \lap u + J V'(|u|^2)u.
$$
\end{itemize}
\end{proposition}

\begin{proof} The fact that the law of $X_\infty(t)$ is $\rho$ at all times is due to the fact that $X_n$ converges almost surely in $S= \mathcal C(\R,H_\varphi)$. Hence for all $t$, $X_n(t)$ converges almost surely towards $X_\infty(t)$ in $H_\varphi$. Since the almost sure convergence implies the convergence in law, we get that the law of $X_\infty$ is the limit of the laws of $X_n(t)$, $\rho_{L_n}$, and hence is $\rho$.

Let us prove that $X_\infty$ is a weak solution to 
$$
\partial_t u = J \lap u - J V'(|u|^2)u.
$$

We have that 
$$
\partial_t X_\infty -J \lap X_\infty
$$
is almost surely the limit in terms of distributions of 
$$
\partial_t X_n - \Pi_n J \Pi_n \lap X_n
$$
where $\Pi_n = \Pi_{N(L_n)}$.

Indeed, let $f$ be a $\mathcal C^\infty$ with compact support test function of $\R^2$. Since $f$ has compact support, for $L_n$ big enough, we get
$$
\Big| \an{f,\Pi_n J \Pi_n X_n} - \an{f,JX_\infty} \Big| \leq \Big| \an{(J - \Pi_n J \Pi_n) f, X_n}\Big| + \Big| \an{ J f, X_n-X_\infty }\Big|
$$
where $\an{\cdot, \cdot}$ is the inner product and $\Pi_n$ when applied to $f$ stands for the Fourier multiplier $\widehat{\Pi_n f}(k) = \eta_n (k) \hat f (k)$ where $\eta_n$ is a $\mathcal C^\infty$ function which is equal to $1$ on $[-N(L_n),N(L_n)]$ and to $0$ outside $[-N(L_n) -\frac1{2L_n},N(L_n) + \frac1{2L_n}]$. Since $X_n$ converges towards $X_\infty$ in $S$, $X_n(t)$ converges towards $X_\infty (t)$ in $H_\varphi$, hence 
$$
\Big| \an{(J - \Pi_n J \Pi_n) f(t), X_n(t)}\Big| \leq \|(J - \Pi_n J \Pi_n) f(t)\| \sup_n \|X_n(t)\|_\varphi \leq \|(J - \Pi_n J \Pi_n) f\| \sup_n \|X_n\|_S
$$
where $\|\cdot\|$ is the norm of the dual of $H_\varphi$. We have
$$
\|(J - \Pi_n J \Pi_n) f(t)\| = \|\varphi^{-1}(x) \an x D^\kappa (J - \Pi_n J \Pi_n) f(t)\|_{L^2}.
$$
As $f(t)$ has a compact support, we get
$$
\|(J - \Pi_n J \Pi_n) f(t)\| \leq \sup_{x\in\textrm{supp }f}\Big( \varphi^{-1}(x) \an x \Big)\| D^\kappa (J - \Pi_n J \Pi_n) f(t)\|_{L^2}.
$$
Since $J - \Pi_n J \Pi_n = (1-\Pi_n) J + \Pi_n J(1-\Pi_n)$ and thanks to Assumption \ref{assum-J}, we have for $\sigma > 0$,
$$
\| D^\kappa (J - \Pi_n J \Pi_n) f(t)\|_{L^2} \lesssim N(L_n)^{-\sigma} \|f(t)\|_{H^{\sigma + 2\kappa}},
$$
from which we deduce,
\begin{multline*}
\Big| \an{f,\Pi_n J_L \Pi_n X_n} - \an{f,JX_\infty} \Big| \leq \sup_{(t,x)\in \textrm{supp }f} \Big( \an t \varphi^{-1}(x) \an x \|f(t)\|_{H^{\sigma + 2\kappa}}\Big) N(L_n)^{-\sigma} + \\
\sup_{(t,x)\in \textrm{supp }f} \Big( \an t \|Jf(t)\|\Big) \|X_n-X_\infty\|_{S}
\end{multline*}
which goes to $0$ when $n$ goes to $\infty$.

Besides, we have 
$$
| V'(|X_\infty|^2)X_\infty - V'(|X_n|^2)X_n| \leq |V'(|X_\infty|^2) | \, |X_n - X_\infty| + \left(\sup_{[|X_\infty|^2,|X_n|^2]}|V''|\right) \, |X_n| (|X_\infty|+|X_n|)|X_\infty - X_n|.
$$
With the hypothesis on $V$, Assumption \ref{assum-V}, we get
$$
| V'(|X_\infty|^2)X_\infty - V'(|X_n|^2)X_n| \lesssim \an{X_\infty}^{r_V}   |X_n - X_\infty| + \Big( \an{X_\infty}^{r_V} + \an{X_n}^{r_V}\Big) \, |X_n| (|X_\infty|+|X_n|)|X_\infty - X_n|.
$$
Therefore, for all weight functions $g$ such that $g=h^r$ with
$$
\|\an t \varphi_1^{-1} D^s h\|_{L^\infty_{x,t}} < \infty,
$$
we have
\begin{multline*}
\|g(x,t)\an{x}^{-1}\varphi \an{t}^{-6} (V'(|X_\infty|^2)X_\infty - V'(|X_n|^2)X_n)\|_{L^1(\R\times \R)} \lesssim \\
\Big( 1+ \|g(x,t)X_\infty^{r_V}\|_{L^2(\R\times \R)} + \|g(x,t)X_\infty^{r_V+2}\|_{L^2(\R\times \R)}+\|g(x,t)X_n^{r_V+2}\|_{L^2(\R\times \R)}\Big) \|\an{x}^{-1}\varphi \an{t}^{-6}(X_\infty-X_n)\|_{L^2(\R\times \R)}.
\end{multline*}
By taking the $L^1$ norm in probability, we get
\begin{multline*}
\|g(x,t)\an{x}^{-1}\varphi \an{t}^{-6} (V'(|X_\infty|^2)X_\infty - V'(|X_n|^2)X_n)\|_{L^1(\Omega \times \R\times \R)} \lesssim\\
\Big( 1+ \|g(x,t)X_\infty^{r_V}\|_{L^2(\Omega \times \R\times \R)} + \|g(x,t)X_\infty^{r_V+2}\|_{L^2(\Omega \times \R\times \R)}+\|g(x,t)X_n^{r_V+2}\|_{L^2(\Omega \times \R\times \R)}\Big) \|\an{x}^{-1}\varphi \an{t}^{-6}(X_\infty-X_n)\|_{L^2(\Omega \times \R\times \R)}.
\end{multline*}

For $r = r_V$ or $r=r_V+2$, we get, using Sobolev's estimates,
$$
\|g(x,t)X_n^{r}\|_{L^2}^2 \leq \|\an t \varphi_1^{-1} D^s h\|_{L^\infty_{x,t}}^{2r} \E \Big(\int \an{t}^{-2}dt \| \varphi_1 D^s X_n \|_{L^2(\R)}^{2r}\Big).
$$
We exchange the integrals in time and probability to get
$$
\|g(x,t)X_n^{r}\|_{L^2}^2 \leq \int \an{t}^{-2}dt \E \Big(\| \varphi_1 D^s X_n \|_{L^2(\R)}^{2r}\Big).
$$
Given the law of $X_n$, this yields
$$
\E \Big(\| \varphi_1 D^s X_n \|_{L^2(\R)}^{2r}\Big) = \int \Big( \| \varphi_1 D^s u \|_{L^2(\R)}^{2r} \Big) d\rho_L(u)\leq C_{2r,s}.
$$
From which we deduce 
\begin{multline*}
\|g(x,t)\an{x}^{-1}\varphi \an{t}^{-2} (V'(|X_\infty|^2)X_\infty - V'(|X_n|^2)X_n)\|_{L^1(\Omega \times \R\times \R)} \lesssim \\
\Big( 1+ C_{2r_V,s} + C_{2r_V+4,s}\Big) \|\an{x}^{-1}\varphi \an{t}^{-2}(X_\infty-X_n)\|_{L^2(\Omega \times \R\times \R)}.
\end{multline*}

For 
$$
\|\an{x}^{-1}\varphi \an{t}^{-6}(X_\infty-X_n)\|_{L^2(\Omega \times \R\times \R)},
$$
we fix some time $t$ and consider 
$$
\|\an{x}^{-1}\varphi (X_\infty(t)-X_n(t))\|_{L^2(\Omega \times \R)}
$$
Following the proof of Proposition \eqref{compball} we get that for all $\varepsilon > 0$, there exists $X,N$ such that for all $n$,
$$
\|\an{x}^{-1}\varphi (X_n - X_n^{X,N})\|_{L^2(\R)} \leq \varepsilon \|\varphi_1 D^s X_n\|_{L^2(\R)}.
$$
We integrate in probability to get 
$$
\|\an{x}^{-1}\varphi (X_n - X_n^{X,N})\|_{L^2\Omega \times \R)} \leq \varepsilon \|\varphi_1 D^s X_n\|_{L^2(\Omega \times \R)} \leq \varepsilon \sqrt{C_{2,s}}.
$$
We recall that $C_{2,s}$ does not depend on $n$.  Hence, we have 
$$
\|\an{x}^{-1}\varphi (X_\infty(t)-X_n(t))\|_{L^2(\Omega \times \R)}\leq \sqrt{C_{2,s}}\varepsilon + \|\an{x}^{-1}\varphi (X_\infty(t)^{X,N}-X_n^{X,N}(t))\|_{L^2(\Omega \times \R)}.
$$
We use the fact that $(X_\infty(t)^{X,N}-X_n^{X,N}(t))$ belongs to a space of finite dimension to get
$$
\|\an{x}^{-1}\varphi (X_\infty(t)^{X,N}-X_n^{X,N}(t))\|_{L^2(\Omega \times \R)} \leq C(T,N) \|X_\infty(t)^{X,N}-X_n^{X,N}(t)\|_\varphi
$$
and finally
$$
\|\an{x}^{-1}\varphi (X_\infty(t)^{X,N}-X_n^{X,N}(t))\|_{L^2(\Omega \times \R)} \leq C_1(T,N) \|X_\infty(t)-X_n(t)\|_\varphi.
$$
Integrating in time yields 
$$
\|\an{x}^{-1}\varphi \an{t}^{-6}(X_\infty-X_n)\|_{L^2( \Omega \times \R\times \R)}^2 \leq C_2\varepsilon + C_1(T,N) \E\Big(  \int \frac{dt}{\an{t}^{12}} \|X_\infty(t)-X_n(t)\|_\varphi^2\Big)
$$
which gives 
$$
\|\an{x}^{-1}\varphi \an{t}^{-6}(X_\infty-X_n)\|_{L^2( \Omega \times \R\times \R)}^2 \leq C_2\varepsilon + C_1(T,N) \E\Big(  \|X_\infty-X_n\|_S^2\Big).
$$
By the dominated convergence theorem, $\E\Big(  \|X_\infty-X_n\|_S^2\Big)$ converges towards $0$. Indeed, Let $R\geq 0$, and let $f_n = \|X_\infty-X_n\|_S^2$, let $g_n = 1_{f_n\leq R} f_n$. We have that $g_n$ converges almost surely towards $0$ and $g_n$ is bounded. Hence, $\E(g_n)$ converges towards $0$ by DCT. Besides, $f_n = g_n + 1_{f_n > R}f_n$ and 
$$
\E(1_{f_n > R}f_n) \leq \sqrt{P(f_n> R)} \E(f_n^2)^{1/2} \leq R^{-1} \E(f_n^2).
$$
Finally, $\E(f_n^2) \lesssim \E( \|X_\infty\|_S^4 + \|X_n\|_S^4)$ is uniformly bounded in $n$. 

From that we deduce that
$$
\|g(x,t)\an{x}^{-1}\varphi \an{t}^{-6} (V'(|X_\infty|^2)X_\infty - V'(|X_n|^2)X_n)\|_{L^1(\Omega \times \R\times \R)}
$$
goes to $0$ when $n$ goes to $\infty$. Since $\chi_L$ goes to $1$ in $\an{x}L^\infty$, we get that
$$
\|g(x,t)\an{x}^{-2}\varphi \an{t}^{-6} (V'(|X_\infty|^2)X_\infty - \chi_{L_n}V'(|X_n|^2)X_n)\|_{L^1(\Omega \times \R\times \R)}
$$
goes to $0$ when $n$ goes to $\infty$, which ensures that almost surely, up to a subsequence, $\chi_{L_n}V'(|X_n|^2)X_n$ converges towards $V'(|X_\infty|^2)X_\infty$ in the norm $\| g \an{x}^{-2} \varphi \an{t}^{-6} \cdot \|_{L^1(\R\times \R)}$. By taking $g$ equal to $1$ on some compact set, we get convergence in $L^1_{loc}(\R\times \R)$. Hence, almost surely, up to a subsequence, and in the sense of distributions
$$
\Pi_n J_L \Pi_n \chi_{L_n}V'(|X_n|^2)X_n) \underset{n\rightarrow \infty}{\longrightarrow} J V'(|X_\infty|^2)X_\infty.
$$

Finally, almost surely, up to a subsequence, we have that
$$
0 = \partial_t X_n + \Pi_nJ \Pi_n\lap X_n - \Pi_nJ\Pi_n \chi_{L_n}V'(|X_n|^2)X_n
$$
converges towards
$$
\partial_t X_\infty + J\lap X_\infty - J V'(|X_\infty|^2)X_\infty
$$
which ensures that almost surely,
$$
\partial_t X_\infty + J\lap X_\infty - J V'(|X_\infty|^2)X_\infty = 0
$$
\end{proof}

\section{Proofs of the estimates and convergence}

We devote this section to prove a uniform-in-L moment bound (Proposition \ref{prop-est}) and the weak convergence of $\rho_L$ towards $\rho$ (Proposition \ref{convprop}). With these two ingredients, Theorem \eqref{teo1} follows from Proposition \ref{prop2}.

\subsection{Estimates}

We recall the assumptions on $\chi_L$, Assumption \ref{assum-chi}: it is a $\mathcal C^\infty$ function such that $\chi_L(x)=1$ if $x\in [-R(L),R(L)]$, $\chi_L(x) = 0$ if $x\notin [-R'(L),R'(L)]$ and $\chi_L(x) \in [0,1]$. And we recall that $R(L)$ has been chosen small enough such that
$$
Z_{L,3} = \E \Big( e^{-\int_R(L)^R(L) V(|\xi(x)|^2)dx }\Big) \geq L^{-1/6},
$$
and that $R'(L)$ has been chosen close enough to $R(L)$ such that $R'(L) - R(L) \leq \frac{1}{CL^{1/2}}$ with $C$ a constant big enough.

\begin{proposition}\label{prop-est} For all $r\geq 2$, all $s<\frac12$, there exists $C_{r,s}$ and a positive, even, decreasing on $\R^+$ map $\varphi_1$ such that for all $L$
$$
\int \Big( \| \varphi_1 D^s u \|_{L^2(\R)}^r \Big) d\rho_L(u)\leq C_{r,s}.
$$
What is more, for all $r\geq 2$ we have 
$$
\int |u(x)|^r d\rho_L(u) \lesssim \varphi_r(x).
$$
\end{proposition}

We divide the proposition into four lemmas.

\begin{lemma}\label{lem-0} 
Let $Z_{L}$, $Z_{L,1}$, $Z_{L,2}$ and $Z_{L,3}$ defined respectively by \eqref{ZL},\eqref{ZL1},\eqref{ZL2} and\eqref{ZL3}. We have \begin{itemize}
\item $$\E\Big(\Big| e^{-\int \chi_L V(|\xi|^2)} - e^{-\int_{-R(L)}^R(L) V(|\xi|^2) }\Big|^2 \Big)\leq Z_{L,3}^6$$ which ensures in particular $Z_{L,2} \geq Z_{L,3}(1-Z_{L,3}^2)$,
\item $$ \E\Big(\Big| e^{-\int \chi_L V(|\xi|^2)} - e^{-\int \chi_L V(|\xi_L|^2) } \Big|^2\Big) \leq Z_{L,3}^4$$ which ensures in particular $ Z_{L,1} \geq Z_{L,3}(1-2Z_{L,3})$,
\item $$ \E\Big(\Big| e^{-\int \chi_L V(|\xi_L|^2)} - e^{-\int \chi_L V(|\xi_L^f|^2) } \Big|^2\Big) \leq Z_{L,3}^4$$ which ensures in particular $ Z_{L} \geq Z_{L,3}(1-3Z_{L,3})$.
\end{itemize}
\end{lemma}

\begin{lemma}\label{lem-0.5} There exists a positive, even, decreasing on $\R^+$ map $\varphi_1$ such that for all $r\geq 2$, all $s<\frac12$, there exists $C_{r,s}$ such that for all $L$
$$
\E\Big(\Big| \frac{e^{-\int \chi_L V(|\xi_L|^2)}}{Z_{L,1}} \| \varphi_1 D^s \xi_L \|_{L^2(\R)}^r  - \frac{e^{-\int \chi_L V(|\xi_L^f|^2)}}{Z_{L}} \| \varphi_1 D^s \xi_L^f \|_{L^2(\R)}^r\Big|\Big) \leq C_{r,s}.
$$
What is more, for all $r\geq 2$, there exists $C_r$ such that for all $x$
$$
\E\Big(\Big| \frac{e^{-\int \chi_L V(|\xi_L|^2)}}{Z_{L,1}} |  \xi_L (x)|^r  - \frac{e^{-\int \chi_L V(|\xi_L^f|^2)}}{Z_{L}} |  \xi_L (x)|^r\Big|\Big) \leq C_{r,}.
$$

\end{lemma}

\begin{lemma}\label{lem-1} There exists a positive, even, decreasing on $\R^+$ map $\varphi_1$ such that for all $r\geq 2$, all $s<\frac12$, there exists $C_{r,s}$ such that for all $L$
$$
\E\Big(\Big| \frac{e^{-\int \chi_L V(|\xi_L|^2)}}{Z_{L,1}} \| \varphi_1 D^s \xi_L \|_{L^2(\R)}^r  - \frac{e^{-\int \chi_L V(|\xi|^2)}}{Z_{L,2}} \| \varphi_1 D^s \xi \|_{L^2(\R)}^r\Big|\Big) \leq C_{r,s}.
$$
What is more, for all $r\geq 2$, there exists $C_r$ such that for all $x$
$$
\E\Big(\Big| \frac{e^{-\int \chi_L V(|\xi_L|^2)}}{Z_{L,1}} | \xi_L (x)|^r  - \frac{e^{-\int \chi_L V(|\xi|^2)}}{Z_{L,2}} |  \xi(x) |^r\Big|\Big) \leq C_{r} \an x.
$$

\end{lemma}

\begin{lemma}\label{lem-2} for all $r\geq 2$, all $s<\frac12$, there exists $C_{r,s}$ and a positive, even, decreasing on $\R^+$ map $\varphi_1$ such that for all $L$
$$
\E\Big(\frac{e^{-\int \chi_L V(|\xi|^2)}}{Z_{L,2}} \| \varphi_1 D^s \xi \|_{L^2(\R)}^r \Big)\leq C_{r,s}.
$$
What is more, for all $r\geq 2$, there exists $C_r$ such that for all $x$
$$
\E\Big(\frac{e^{-\int \chi_L V(|\xi|^2)}}{Z_{L,2}} | \xi(x) |^r \Big)\leq C_{r}\varphi_r(x).
$$
\end{lemma}

\begin{proof}[Proof of Lemma \ref{lem-0}. ]We have 
$$
I = \E\Big(\Big| e^{-\int \chi_L V(|\xi|^2)} - e^{-\int_{-R(L)}^R(L) V(|\xi|^2) }\Big|^2 \Big)\leq \E\Big( \Big| \int_{\R} |\chi_L - 1_{[-R(L),R(L)]}| V(|\xi|^2) \Big|^2 \Big)
$$
and exchanging the order of integration we get
$$
I \leq \int dx dy |\chi_L(x) - 1_{[-R(L),R(L)]}(x)| |\chi_L(y) - 1_{[-R(L),R(L)]}(y)| \E(V(|\xi(x)|^2) V(|\xi(y)|^2))
$$
and since $V(|\xi(x)|^2) \leq \an{\xi(x)}^{r_V}$ and since the law of $\xi$ is invariant by translation, we get that
$$
\E(V(|\xi(x)|^2) V(|\xi(y)|^2)) \leq \E( \an{\xi(x)}^{2r_V})^2
$$
is less that a constant depending only on $V$. Hence
$$
I \lesssim \Big( \int |\chi_L(x) - 1_{[-R(L),R(L)]}(x)|\Big)^2
$$
and given the Assumptions \ref{assum-chi} on $\chi_L$ this yields
$$
I \lesssim |R'(L) - R(L)|^2 \leq c L^{-1}
$$
which gives the first result assuming that the constant $C$ in the definition on $R'(L)= R(L) + \frac1{C\sqrt L}$ has been chosen big enough.

We also have 
$$
II =\E\Big(\Big| e^{-\int \chi_L V(|\xi|^2)} - e^{-\int \chi_L V(|\xi_L|^2) } \Big|^2\Big) \leq \E\Big( \int |\chi_L (V(|\xi|^2) - V(|\xi_L|^2))|^2 \Big).
$$
With the assumption on $V'$, Assumption \ref{assum-V}, we get that
$$
\sqrt{II} \leq \int \chi_L \|\an{\xi(x)}^{r_V+1} + \an{\xi_L(x)}^{r_V+1}\|_{L^4_{\textrm{proba}}} \|\xi-\xi_L\|_{L^4_{\textrm{proba}}}
$$
where we recall that the space $L^p_{\textrm{proba}}$ is short for the $L^p$ space of the probabilistic space where the Gaussian process $W$ is defined.

Thanks to Proposition \ref{prop-prob1}, we have that
$$
\|\xi-\xi_L\|_{L^4_{\textrm{proba}}} \lesssim \an{x} L^{-1}
$$
and that 
$$
\|\an{\xi(x)}^{r_V+1} + \an{\xi_L(x)}^{r_V+1}\|_{L^4_{\textrm{proba}}}
$$
is uniformly bounded in $x$ and $L$. Therefore,
$$
\sqrt{II} \lesssim L^{-1/2} \int \chi_L \an{x}.
$$
Choosing $R(L)$ small enough such that $\int\chi_L \an{x} \leq cL^{1/6}$  with $c$ small enough we get 
$$
II \leq L^{-2/3}= Z_{L,3}^4.
$$

For 
$$
III = \E\Big(\Big| e^{-\int \chi_L V(|\xi_L|^2)} - e^{-\int \chi_L V(|\xi_L^f|^2) } \Big|^2\Big)
$$
we have 
$$
\sqrt{III} \leq \int \chi_L \|\an{\xi_L^f(x)}^{r_V+1} + \an{\xi_L(x)}^{r_V+1}\|_{L^4_{\textrm{proba}}} \|\xi_L^f-\xi_L\|_{L^4_{\textrm{proba}}}.
$$
We have that $\xi_L^f-\xi_L$ is a Gaussian hence 
$$
\|\xi_L^f-\xi_L\|_{L^4_{\textrm{proba}}} \lesssim \|\xi_L^f-\xi_L\|_{L^2_{\textrm{proba}}}.
$$
The $L^2$ norm to the square is given by
\begin{equation}\label{relest}
\|\xi_L^f-\xi_L\|_{L^2_{\textrm{proba}}}^2 = \sum_{k\in Z , |k|/L > N(L)} \frac1{1+ \frac{k^2}{L^2}} \frac1{L} \lesssim N(L)^{-1/2}  \int \frac{dy}{(1+y^2)^{3/4}}.
\end{equation}
What is more,
$$
\|\an{\xi_L^f(x)}^{r_V+1} + \an{\xi_L(x)}^{r_V+1}\|_{L^4_{\textrm{proba}}}
$$
is uniformly bounded in $x$ and $L$. Therefore, with the choice of $N(L)$, Assumption \ref{assum-N}, we have
$$
\sqrt{III} \lesssim L^{-1} \int \chi_L .
$$
Choosing $R(L)$ small enough such that $\int\chi_L  \leq c L^{1/2}$  with $c$ small enough we get 
$$
III \leq L^{-1}\leq Z_{L,3}^4
$$
which concludes the proof.
\end{proof}

\begin{proof}[Proof of Lemma \ref{lem-0.5}. ] Let $\varphi_1\in L^1\cap L^\infty$. and let
$$
A= \E\Big(\Big| \frac{e^{-\int \chi_L V(|\xi_L|^2)}}{Z_{L,1}} \| \varphi_1 D^s \xi_L \|_{L^2(\R)}^r  - \frac{e^{-\int \chi_L V(|\xi_L^f|^2)}}{Z_{L}} \| \varphi_1 D^s \xi_L^f \|_{L^2(\R)}^r\Big|\Big) .
$$
The proof of this lemma and the next one are new compared to the other proofs. They rely on the fact that by choosing appropriate $N(L), R(L)$, the measure $\rho_L$ converges towards $\rho$. 

We have
$$
A \leq A_1 + A_2
$$
with 
$$
A_1 = \E\Big( \Big| \frac{e^{-\int \chi_L V(|\xi_L|^2)}}{Z_{L,1}} -\frac{e^{-\int \chi_L V(|\xi_L^f|^2)}}{Z_{L}}\Big| \|\varphi_1 D^s \xi_L \|_{L^2(\R)}^r\Big)
$$
and
$$
A_2 = \E\Big( \frac{e^{-\int\chi_L V(|\xi_L^f|^2)}}{Z_{L}}\Big| \|\varphi_1 D^s\xi_L\|_{L^2}^r -\|\varphi_1 D^s \xi_L^f\|_{L^2}^r\Big|\Big).
$$

By H\"older's inequality, we have 
$$
A_1 \leq \E\Big( \Big|  \frac{e^{-\int \chi_L V(|\xi_L|^2)}}{Z_{L,1}} -\frac{e^{-\int \chi_L V(|\xi_L^f|^2)}}{Z_{L}}\Big|^2\Big)^{1/2} \E(\|\varphi_1 D^s \xi_L\|_{L^2}^{2r})^{1/2}.
$$
Since $\varphi_1$ is bounded, as long as in $L^1$ and $s<\frac12$, we have that
$$
\E(\|\varphi_1 D^s \xi_L\|_{L^2}^{2r})^{1/2}
$$
is uniformly bounded in $L$ (but not in $r,s$). Hence,
$$
A_1 \lesssim \frac1{Z_{L,1}} \E\Big( \Big|  e^{-\int \chi_L V(|\xi_L|^2)}-e^{-\int \chi_L V(|\xi_L^f|^2)}\Big|^2\Big)^{1/2} + \E\Big( e^{-2 \int\chi_L V(|\xi_L^f|^2)}\Big)^{1/2} \Big| \frac1{Z_{L,1}}-\frac1{Z_{L}}\Big|.
$$
Thanks to Lemma \ref{lem-0}, we get
$$
A_1 \lesssim \frac{Z_{L,3}}{1-2Z_{L,3}} + \frac{Z_{L,3}^{1/2}}{(1-3Z_{L,3})^{1/2}(1-2Z_{L,3})},
$$
which goes to $0$ as $L$ goes to $\infty$ and hence is bounded.

By H\"older's inequality, we have 
$$
A_2 \leq \frac{\E(e^{-2\int\chi_L V(|\xi_L^f|^2) })^{1/2}}{Z_{L}} \Big( \E(\|\varphi_1 D^s \xi_L^f \|_{L^2}^{4(r-1)})^{1/4} +\E(\|\varphi_1 D^s \xi_L \|_{L^2}^{4(r-1)})^{1/4}\Big) \E( \|\varphi_1 D^s(\xi_L-\xi_L^f)\|_{L^2}^4)^{1/4}.
$$
We have that
$$
 \E(\|\varphi_1 D^s \xi_L^f \|_{L^2}^{4(r-1)})^{1/4}
$$
is uniformly bounded in $L$ (but not in $r$, $s$) as long as $\varphi_1$ is in $L^1$. Therefore,
$$
A_2 \lesssim (Z_{L})^{-1/2} \E( \|\varphi_1 D^s(\xi_L-\xi_L^f)\|_{L^2}^4)^{1/4}.
$$
We have that
$$
\E( \|\varphi_1 D^s(\xi_L-\xi_L^f)\|_{L^2}^4)^{1/4} = \|\varphi_1 D^s(\xi_L-\xi_L^f)\|_{L^4_{\textrm{proba}},L^2(\R)}
$$
and by Minkowski's inequality, since $4\geq 2$, we can exchange the norms to get
$$
\E( \|\varphi_1 D^s(\xi_L-\xi_L^f)\|_{L^2}^4)^{1/4} \leq \|\varphi_1 D^s(\xi_L-\xi_L^f)\|_{L^2(\R),L^4_{\textrm{proba}}}.
$$
Given $\xi_L$ and $\xi_L^f$, we have that for all $x$ (recall \eqref{relest})
$$
\|D^s(\xi_L-\xi_L^f)(x)\|_{L^4_{\textrm{proba}}} \lesssim  N(L)^{-1/4+s/2} \Big( \int \frac{dy}{(1+y^2)^{3/4-s/2}}\Big)^{1/4}
$$
and thus
$$
\E( \|\varphi_1 D^s(\xi-\xi_L)\|_{L^2}^4)^{1/4} \lesssim N(L)^{-1/4+s/2} \|\varphi_1 \|_{L^2}.
$$
Hence, as long as $\varphi_1 $ is in $L^2$ we have 
$$
A_2 \lesssim N(L)^{-1/4+s/2} Z_{L,3}^{-1/2}(1-3Z_{L,3})^{-1/2} 
$$
and given the estimate on $Z_{L,3}$, Assumption \ref{assum-R}, and Assumption \ref{assum-N}, we have 
$$
A_2 \lesssim (1-3Z_{L,3})^{-1/2}
$$
which is bounded, and this concludes the proof of the first inequality. The proof of the second inequality is similar except that one has to replace $\| \varphi_1 D^s \xi_L \|_{L^2(\R)}$ with $|\xi_L(x)|$ and $\| \varphi_1 D^s \xi_L^f \|_{L^2(\R)}$ with $|\xi_L^f(x)|$. The fact that the bound does not depend on $x$ is due to the invariance of the laws of $\xi_L$ and $\xi_L^f$ and their difference under space translations.

\end{proof}

\begin{proof}[Proof of Lemma \ref{lem-1}. ] Let 
$$
A= \E\Big(\Big| \frac{e^{-\int \chi_L V(|\xi_L|^2)}}{Z_{L,1}} \| \varphi_1 D^s \xi_L \|_{L^2(\R)}^r  - \frac{e^{-\int \chi_L V(|\xi|^2)}}{Z_{L,2}} \| \varphi_1 D^s \xi \|_{L^2(\R)}^r\Big|\Big) .
$$

We have
$$
A \leq A_1 + A_2
$$
with 
$$
A_1 = \E\Big( \Big| \frac{e^{-\int \chi_L V(|\xi_L|^2)}}{Z_{L,1}} -\frac{e^{-\int \chi_L V(|\xi|^2)}}{Z_{L,2}}\Big| \|\varphi_1 D^s \xi \|_{L^2(\R)}^r\Big)
$$
and
$$
A_2 = \E\Big( \frac{e^{-\int\chi_L V(|\xi_L|^2)}}{Z_{L,1}}\Big| \|\varphi_1 D^s\xi_L\|_{L^2}^r -\|\varphi_1 D^s \xi\|_{L^2}^r\Big|\Big).
$$

By H\"older's inequality, we have 
$$
A_1 \leq \E\Big( \Big|  \frac{e^{-\int \chi_L V(|\xi_L|^2)}}{Z_{L,1}} -\frac{e^{-\int \chi_L V(|\xi|^2)}}{Z_{L,2}}\Big|^2\Big)^{1/2} \E(\|\varphi_1 D^s \xi\|_{L^2}^{2r})^{1/2}.
$$
As long as $\varphi_1$ is in $L^1$ and $s<\frac12$, we have that
$$
\E(\|\varphi_1 D^s \xi\|_{L^2}^{2r})^{1/2}
$$
is finite. Hence,
$$
A_1 \lesssim \frac1{Z_{L,1}} \E\Big( \Big|  e^{-\int \chi_L V(|\xi_L|^2)}-e^{-\int \chi_L V(|\xi|^2)}\Big|^2\Big)^{1/2} + \E\Big( e^{-2 \int\chi_L V(|\xi|^2)}\Big)^{1/2} \Big| \frac1{Z_{L,1}}-\frac1{Z_{L,2}}\Big|.
$$
Thanks to Lemma \ref{lem-0}, we get
$$
A_1 \lesssim \frac{Z_{L,3}}{1-2Z_{L,3}} + \frac{Z_{L,3}^{1/2}}{(1-2Z_{L,3})(1-Z_{L,3}^2)},
$$
which goes to $0$ as $L$ goes to $\infty$ and hence is bounded.

By H\"older's inequality, we have 
$$
A_2 \leq \frac{\E(e^{-2\int\chi_L V(|\xi_L|^2) })^{1/2}}{Z_{L,1}} \Big( \E(\|\varphi_1 D^s \xi_L \|_{L^2}^{4(r-1)})^{1/4} +\E(\|\varphi_1 D^s \xi \|_{L^2}^{4(r-1)})^{1/4}\Big) \E( \|\varphi_1 D^s(\xi-\xi_L)\|_{L^2}^4)^{1/4}.
$$
We have that
$$
 \E(\|\varphi_1 D^s \xi_L \|_{L^2}^{4(r-1)})^{1/4}
$$
is uniformly bounded in $L$ as long as $\varphi_1$ is in $L^1$. Therefore,
$$
A_2 \lesssim (Z_{L,1})^{-1/2} \E( \|\varphi_1 D^s(\xi-\xi_L)\|_{L^2}^4)^{1/4}.
$$
We have that
$$
\E( \|\varphi_1 D^s(\xi-\xi_L)\|_{L^2}^4)^{1/4} = \|\varphi_1 D^s(\xi-\xi_L)\|_{L^4_{\textrm{proba}},L^2(\R)}
$$
and by Minkowski's inequality, since $4\geq 2$, we can exchange the norms to get
$$
\E( \|\varphi_1 D^s(\xi-\xi_L)\|_{L^2}^4)^{1/4} \leq \|\varphi_1 D^s(\xi-\xi_L)\|_{L^2(\R),L^4_{\textrm{proba}}}.
$$
Given $\xi$ and $\xi_L$, due to Proposition \ref{prop-prob1} we have that for all $x$ 
$$
\|D^s(\xi-\xi_L)(x)\|_{L^4_{\textrm{proba}}} \lesssim \an x L^{-1/2}
$$
and thus
$$
\E( \|\varphi_1 D^s(\xi-\xi_L)\|_{L^2}^4)^{1/4} \lesssim \frac1{L^{1/2}}\|\varphi_1 \an x\|_{L^2}.
$$
Hence, as long as $\varphi_1 \an x$ is in $L^2$ we have 
$$
A_2 \lesssim \frac1{L^{1/2}Z_{L,3}^{1/2}(1-2Z_{L,3})^{1/2}} 
$$
and given the estimate on $Z_{L,3}$, we have 
$$
A_2 \lesssim \frac1{L^{5/12}(1-2Z_{L,3})^{1/2}}
$$
which goes to $0$ as $L$ goes to $\infty$ and hence $A_2$ is bounded and this concludes the proof of the first inequality. Again, the proof of the second inequality is similar except that one has to replace $\| \varphi_1 D^s \xi_L \|_{L^2(\R)}$ with $|\xi_L(x)|$ and $\| \varphi_1 D^s \xi \|_{L^2(\R)}$ with $|\xi(x)|$. The fact that the upper bound is given by $\langle x\rangle$ is due to the invariance of the laws of $\xi_L$ and $\xi$ under space translations and Proposition \ref{prop-prob1}.
\end{proof}

\begin{proof}[Proof of Lemma \ref{lem-2}. ]Let 
$$
B = 
\E\Big(\frac{e^{-\int \chi_L V(|\xi|^2)}}{Z_{L,2}} \| \varphi_1 D^s \xi \|_{L^2(\R)}^r\Big) 
$$
We have 
$$
B \leq B_1 + B_2
$$
with 
$$
B_1 = 
\E\Big(\Big|\frac{e^{-\int \chi_L V(|\xi|^2)}}{Z_{L,2}}-  \frac{e^{-\int_{-R(L)}^R(L) V(|\xi|^2)}}{Z_{L,3}}\Big| \| \varphi_1 D^s \xi \|_{L^2(\R)}^r\Big) 
$$
and 
$$
B_2 = \E\Big(\frac{e^{-\int_{-R(L)}^R(L) V(|\xi|^2)}}{Z_{L,3}} \| \varphi_1 D^s \xi \|_{L^2(\R)}^r\Big) .
$$
By H\"older's inequality and for the same reasons as in the proof of Lemma \ref{lem-1}, we have 
$$
B_1 \lesssim \frac1{Z_{L,2}}\E\Big(\Big|e^{-\int \chi_L V(|\xi|^2)}-  e^{-\int_{-R(L)}^R(L) V(|\xi|^2)}\Big|^2\Big)^{1/2} + \E \Big( e^{-2\int_{-R(L)}^R(L) V(|\xi|^2)}\Big) \frac{|Z_{L,2} - Z_{L,3}|}{Z_{L,2} Z_{L,3}}.
$$
From Lemma \ref{lem-0}, we get
$$
B_1 \lesssim \frac{Z_{L,3}}{1-Z_{L,3}^2} + \frac{Z_{L,3}^{1/2}}{1-Z_{L,3}^2}
$$
which is uniformly bounded in $L$ as the RHS above goes to $0$ as $L$ goes to $\infty$.

For $B_2$, we have 
$$
\|\varphi_1 D^s u \|_{L^2}^2 \leq \sum_{n\in \Z} a_n^2 \|D^s u \|_{L^2([n,n+1])}^2
$$
with $a_n = \sup_{[n,n+1]}\varphi_1$. We also have that $ \|D^s u \|_{L^2([n,n+1])}^2$ can be described as 
$$
 \|D^s u \|_{L^2([n,n+1])}^2 = \| u \|^2_{L^2([n,n+1])} + \int_{[n,n+1]^2} dx dy \frac{|u(x) -u(y)|^2}{|x-y|^{1+2s}}
$$
and by symmetry in $x$ and $y$
$$
 \|D^s u \|_{L^2([n,n+1])}^2 = \| u \|^2_{L^2([n,n+1])} +2 \int_{[n,n+1]^2} 1_{|x|\geq |y|}dx dy \frac{|u(x) -u(y)|^2}{|x-y|^{1+2s}}.
$$

Besides, we have with 
$$
d\rho_{L,3} (u) = \frac{e^{-\int_{-R(L)}^R(L)V(|u|^2)}}{Z_{L,3}}d\mu(u),
$$
$$
B_2^{1/r} = \|\varphi_1 D^s u \|_{L^r_{\rho_{L,3}},L^2(\R)}.
$$
We use the description of $\|\varphi_1 D^s u \|_{L^2}$ to get
$$
B_2^{2/r} \leq \|\sum a_n^2 \|D^su\|_{L^2[n,n+1]}^2\|_{L^{r/2}_{\rho_{L,3}}}.
$$
Since $r\geq 2$, by the triangle inequality, we get
$$
B_2^{2/r} \leq \sum a_n^2 \|D^su\|_{L^{r}_{\rho_{L,3}},L^2[n,n+1]}^2
$$
and by using the description of $\|D^s u\|_{L^2([n,n+1])}$,
$$
B_2^{2/r} \lesssim \sum a_n^2 \Big( \|u\|_{L^{r}_{\rho_{L,3}},L^2([n,n+1])}^2 + 2\|\tilde u\|_{L^r_{\rho_{L,3}},L^2([n,n+1]^2)}^2\Big)
$$
where $\tilde u(x,y) = 1_{|x|\geq |y|}\frac{|u(x) - u(y)|}{|x-y|^{1/2+ s}}$.

By Minkowski inequality, since $r\geq 2$, we can exchange the norm in probability and the one in space to get
$$
B_2^{2/r} \leq \sum_n a_n^2 \Big( \|u\|_{L^2([n,n+1],L^{r}_{\rho_{L,3}})}^2 + 2\|\tilde u\|_{L^2([n,n+1]^2,L^r_{\rho_{L,3}})}^2\Big).
$$
Since $\rho_{L,3}$ is a probability measure, we have that $L^r_{d\rho_{L,3}}$ is continuously embedded in $L^{r'}_{d\rho_{L,3}}$ for $r\leq r'$. We take $r'\geq r_s,r$, with $r_s$ as in Proposition \ref{prop-prob2}.

Thanks to Proposition \ref{prop-prob2}, there exists $\varphi_r$ such that 
$$
\|u(x)\|_{L^r_{\rho_{L,3}}}^{r'}  \leq \varphi_{r'}(|x|)
$$
and
$$
\|\tilde u(x,y)\|_{L^r_{\rho_{L,3}}}^{r'}  \leq \varphi_{r'}(|x|)|x-y|^{-r'/2+1}.
$$
This is due to Feynman-Kac's integrals and the dependence in $x$ is due to different rates of point-wise convergence in terms of $x$.

Therefore, we have by taking the previous inequalities to the power $2/r'$
$$
\|\tilde u(x,y)\|^2_{L^2([n,n+1]^2,L_{\rho_{L,3}}^r)} \leq \int_{n}^{n+1} \varphi_{r'}^{2/r'}(|x|) \int_{n}^{n+1} |x-y|^{-1+2/r'}dydx,
$$
and since $-1+\frac2{r'} > -1$, we get
$$
B_2^{2/r'} \lesssim \sum_n a_n^2 \Big( \|\varphi_{r'}^{1/r'}\|_{L^2([n,n+1]}^2 + 2\|\varphi_{r'}^{1/r'}\|_{L^2([n,n+1])}^2\Big).
$$
Choosing $\varphi_1$ small enough such that the series converges, and positive, even, decreasing on $\R^+$, we get the first inequality.
For the second inequality, we use the same decomposition with $B_1$ and $B_2$ only by replacing $\| \varphi_1 D^s \xi \|_{L^2(\R)}$ with $|\xi(x)|$. The term $B_1$ can be treated in exactly the same way; the estimate on the term $B_2$ is the first result of Proposition \ref{prop-prob2}.
\end{proof}

\subsection{Convergence}

\begin{proposition}\label{convprop} The family $(\rho_L)_L$ converges weakly in $H_\varphi$ towards $\rho$ when $L$ goes to $\infty$.\end{proposition}

\begin{proof} Let $F$ be a bounded, Lipschitz continuous function on $S$.

We have 
$$
\Big| \E_{\rho}(F) - \E_{\rho_L} (F)\Big| \leq I + II + III + IV
$$
with
\begin{eqnarray*} 
I & =& \Big| \E_{\rho}(F) - \E_{\rho_{L,3}} (F)\Big| \\
II &=& \Big| \E_{\rho_{L,3}}(F) - \E_{\rho_{L,2}} (F)\Big| \\
III &=& \Big| \E_{\rho_{L,2}}(F) - \E_{\rho_{L,1}} (F)\Big| \\
IV &=& \Big| \E_{\rho_{L,1}}(F) - \E_{\rho_L} (F)\Big|.
\end{eqnarray*}

We have that $I$ goes to $0$ when $L$ goes to $\infty$ by Feynman-Kac theory.

We have 
$$
II \leq \E \Big( |F(\xi)| \Big|  \frac{e^{-\int_R(L)^R(L) V(|\xi|^2)}}{Z_{L,3}} - \frac{e^{- \int \chi_L V(|\xi|^2)}}{Z_{L,2}}\Big|\Big)
$$
which thanks to Lemma \ref{lem-0} and the fact that $F$ is bounded, satisfies
$$
II \leq C_F Z_{L,3} 
$$
where $C_F$ is a constant depending only on $F$ and hence goes to $0$.

We have 
$$
III \leq \E \Big( \Big| F(\xi) \frac{e^{- \int \chi_L V(|\xi|^2)}}{Z_{L,2}} - F(\xi_L) \frac{e^{- \int \chi_L V(|\xi_L|^2)}}{Z_{L,1}}\Big|\Big).
$$
Since $F$ is bounded and Lipschitz continuous we have that 
$$
III \leq C_F \Big( \Big|  \frac{e^{- \int \chi_L V(|\xi|^2)}}{Z_{L,2}} - \frac{e^{- \int \chi_L V(|\xi_L|^2)}}{Z_{L,1}}\Big|\Big) + C_F Z_{L,2}^{-1/2} \E ( \|\xi_L - \xi\|_\varphi^2)^{1/2}.
$$
The norm of $H_\varphi$ is weak enough to get 
$$
\|\xi_L - \xi\|_\varphi \leq \|\an{x}^{-2} (\xi_L - \xi)\|_{L^2}
$$
from which we deduce
$$
\E ( \|\xi_L - \xi\|_\varphi^2)^{1/2} \lesssim L^{1/2}.
$$
Since $Z_{L,2}^{-1/2} \sim L^{1/12}$, and by Lemma \ref{lem-0}, we get that $III$ goes to $0$ when $L$ goes to $\infty$.

Finally,
$$
IV \leq \E \Big( \Big| F(\xi_L) \frac{e^{- \int \chi_L V(|\xi_L|^2)}}{Z_{L,1}} - F(\xi_L^f) \frac{e^{- \int \chi_L V(|\xi_L^f|^2)}}{Z_{L}}\Big|\Big).
$$
Since $F$ is bounded and Lipschitz continuous we have that 
$$
IV \leq C_F \Big( \Big|  \frac{e^{- \int \chi_L V(|\xi_L|^2)}}{Z_{L,1}} - \frac{e^{- \int \chi_L V(|\xi_L^f|^2)}}{Z_{L}}\Big|\Big) + C_F Z_{L,1}^{-1/2} \E ( \|\xi_L - \xi_L^f\|_\varphi^2)^{1/2}.
$$
We have
$$
\E ( \|\xi_L - \xi_L^f\|_\varphi^2)^{1/2} \leq \E ( \|\an{x}^{-1}(\xi_L - \xi_L^f)\|_{L^2}^2)^{1/2} \lesssim N(L)^{-1/4} \leq L^{-1}.
$$
Since $Z_{L,1}^{-1/2} \leq L^{1/12}$, and by Lemma \ref{lem-0}, we get that $IV$ goes to $0$ when $L$ goes to $\infty$.\end{proof}

\appendix

\section{Variable coefficients equations}\label{appvar}

As mentioned in the introduction (see Remark \ref{varcoeffrk}), we can generalize Theorem \ref{teo1} to include also the case of asymptotically flat variable coefficients. Let us assume the following
\begin{assum}\label{assum-a}
Let $a(x)$ be $\mathcal{C}^3$ positive map such that there exist constants $C\in\R$ and $\gamma>1$ such that
$$
a(x)\leq C\langle x\rangle ^{-\gamma}.
$$
\end{assum}

We consider the equation
\begin{equation}\label{unicorn-eq}
i\partial_t u=-\partial_x((1+a)\partial_x u)+V'(|u|^2 u). 
\end{equation}
with $V$ satisfying assumptions \eqref{assum-V}. Then, there should exist a non-trivial measure $\rho$ (independent from $t$), a probability space $(\Omega,\mathcal A, P)$ and a random variable $X_\infty$ with values in $\mathcal C( \R, \mathcal D')$ such that \begin{itemize}
\item for all $t\in \R$, the law of $X_\infty(t)$ is $\rho$,
\item $X_\infty$ is a weak solution of \eqref{unicorn-eq}.
\end{itemize}

The idea is the following. We introduce the change of variable $y=\Phi(x)$ with $\Phi'(x)=\frac1{1+a(x)}$ for every $x$. Then we set $v(y)=u\circ\Phi^{-1}(y)$ so that $v$ satisfies, for $u$ solution of \eqref{unicorn-eq}
$$
i\partial_tv=\frac{-1}{(1+a)\circ\Phi^{-1}(y)}\partial^2_yv+V'(|v|^2)v.
$$
We then get
$$
\partial_t v=J\grad_{\overline{v}}H(v)
$$
with $J=\frac{i}{(1+a)\circ\Phi^{-1}}$ skew-symmetric and with the Hamiltonian given by
$$
H(v)=\frac12\int \overline{v}(-\lap)v+\int (1+a)\circ \Phi^{-1}(y) V(|v|^2).
$$
The difficulty is now that $V$ is replaced by
$$
(1+a)\circ \Phi^{-1}(y)V(|v|^2)
$$
which depends on $y$. Anyway we can write
$$
H=H_0+H_{pert}
$$
with 
$$
H_0(v)=\frac12\int \overline{v}(-\lap)v+\int V(|v|^2) \textrm{ and } H_{pert} = \int a\circ \Phi^{-1}(y) V(|v|^2).
$$
Notice that $H_0$ falls within the assumptions of Theorem \ref{teo1} and therefore defines, in the sense we have seen above, an invariant measure $\rho$ given by
$$
d\rho(u)=\lim_{R(L)\rightarrow \infty}\frac{e^{-\int\chi_L(y)\Phi^{-1}(y) V(|\xi_L^f(y)|^2dy}}{Z_L}d\mu_L.
$$ On the other hand, notice that $a\circ\Phi^{-1}(y)$ is positive and such that 
$$
\left|a\circ \Phi^{-1}(y)\right|\lesssim \langle y\rangle^{-\gamma};
$$
therefore, $H_{pert}=\int a\circ \Phi^{-1}(y)V(|v|^2)$ can be seen as a perturbative term, as $H_{pert}$ is $\rho$ a-s well-defined and $e^{-H_{pert}}\in L^1_\rho$. The proof of Theorem \ref{teo1} can then be reproduced in this new setting. Indeed, the approaching equations are perturbations of the ones in the setting of Theorem \ref{teo1}:
$$
\partial_t v=\Pi_{N(L)}\frac{i}{(1+a)\circ\Phi^{-1}}\Pi_{N(L)}\grad_{\overline u} H_L(u)
$$
(compare with \eqref{fineq}), and the corresponding approached measures are perturbative as well.
Taking $a$ continuous is sufficient to ensure that the image measure by $\Phi$ lives on $\mathcal C^{s}$ for $s<\frac12$ as well as the measure $\rho$. One issue comes from the fact that the solution is a weak solution, in the sense of distribution. Actually, it is a bit better than this. The only problem comes from the convergence of the linear part since the non linear part converges in $L^1_{t,x}$ locally. In the case where $J = i$, in the linear part, we need to be able to lose two derivatives because of the equation and two because of the norm in which we have convergence, wich means that $\Phi$ should send $\mathcal C^4$ maps to $\mathcal C^4$ maps, and this requires that $a$ is $\mathcal C^3$.

\bibliographystyle{amsplain}
\bibliography{CCM_finalRevision} 
\nocite{*}

\end{document}